\documentclass[11pt]{article}

\usepackage{csquotes}
\usepackage[centertags]{amsmath}
\usepackage{amssymb}
\usepackage{amsthm}
\usepackage{amsfonts}
\usepackage[margin=3cm]{geometry}
\usepackage[square,numbers]{natbib}
\usepackage{graphicx}
\usepackage{subcaption}
\usepackage{authblk}
\usepackage{color}
\usepackage{algorithm}
\usepackage{algpseudocode}
\usepackage{setspace}
\usepackage{titlesec}
\usepackage{graphicx}
\usepackage{booktabs}
\usepackage{url}
\usepackage{xcolor}
\usepackage{textcase}

\newtheorem{theorem}{Theorem}

\newtheorem{proposition}{Proposition}

\theoremstyle{definition}

\theoremstyle{remark}
\newtheorem{remark}{Remark}

\titleformat{\section}
  {\normalfont\normalsize\bfseries}{\thesection}{1em}{}

\begin{document}
\small
\onehalfspacing

\title{Optimal Allocation of Embedding Dimensions under Finite-Sample Constraints}
\author{Vasileios E. Papageorgiou}
\affil{National and Kapodistrian University of Athens, Department of Mathematics, Panepistemiopolis, Athens 15784, Greece

Email: vpapageor@math.uoa.gr}
\date{}
\maketitle

\begin{abstract}
\noindent The embedding dimension of categorical predictors is usually selected through heuristic tuning, although it directly affects model complexity, approximation quality, and finite-sample generalization. This paper formulates embedding dimension selection as a constrained allocation problem. The main contribution is to show that embedding capacity can be allocated across heterogeneous categorical predictors according to an explicit approximation–estimation tradeoff. We characterize approximation error through the singular-value tail of the latent category representation, while estimation error increases with total embedding complexity. Under a fixed global embedding budget and a tractable approximation model, this leads to a closed-form allocation rule in which the dimension assigned to each predictor is proportional to the square root of its approximation value relative to its parameter cost. Simulation experiments support the proposed approximation–estimation interpretation and show that the allocation rule improves budget efficiency relative to standard uniform and cardinality-based heuristics, particularly when the budget is binding and predictor heterogeneity is substantial. A real-data healthcare application further shows improvements in predictive accuracy and probabilistic calibration. Overall, the results establish embedding-dimension allocation as a principled finite-sample optimization problem rather than a purely heuristic modeling choice.
\end{abstract}


\section{Introduction}\label{sec:intro}

The growing role of data-driven methods in operations research has increased the need for predictive models that are both flexible and statistically well-structured \citep{TaLiu2020,Zhang2024abc,Pinnaka2024}. This need is especially pronounced in applications involving categorical predictors, where learned representations are used to capture structured heterogeneity in the data. In these problems, the observed covariates may not fully capture systematic variation associated with the categorical structure, so an appropriate representation of categorical effects can be important for prediction, calibration and statistical efficiency \citep{Poziomska2026,Babar2024,Kamm2023}. When this structure is ignored, one effectively imposes a common model across systematically different categories, which may reduce predictive accuracy, distort the learned representation and introduce bias through omitted category effects \citep{Mary-Huard2007}.

Embeddings provide a natural way to represent categorical predictors through low-dimensional vectors learned from data. In electronic health records, for example, patient and concept embeddings have been used to capture clinically meaningful heterogeneity and characterize temporal progression patterns \citep{Xian2025,Wen2025}. More broadly, embeddings can outperform manually specified encodings when they extract task relevant structure from complex categorical inputs \citep{Liu2024b,Hancock2020,Dong2020}. Moreover, \citet{Li2020d} propose a cross-field embedding method that leverages higher-order co-occurrence patterns among categorical variables, while \citet{Mumtaz2022} develop hierarchy-based semantic embeddings that incorporate prior domain knowledge through concept hierarchies, including low-resource settings and multi-valued categorical features. Their effectiveness, however, depends critically on the embedding dimension. If it is too small, important structure may be missed. If it is too large, the model may incur unnecessary estimation cost and lower statistical efficiency.

This tradeoff connects the problem to a central theme in optimization, namely the allocation of limited resources under structural, statistical and computational constraints. In high-dimensional settings, such constraints often appear through the curse of dimensionality, which is well-known in areas such as dynamic programming and scheduling and has motivated the use of structure exploiting and dimension reduction methods to maintain tractability \citep{Barde2025,Shi2026}. Recent work shows that structural properties, such as multimodularity and separability, can substantially reduce the effective decision space while preserving optimality or near optimality \citep{Barde2025,Shi2026}. In parallel, statistical learning and data-driven optimization have emphasized the value of low-dimensional representations for improving predictive performance and computational efficiency  \citep{Jiang2026,Clerici2026,Papageorgiou2026a,Papageorgiou18062026}. Closely related is the literature on sufficient dimension reduction and subspace estimation, which studies how to identify low-dimensional structures that preserve predictive information \citep{Zeng2024,Zhang2026}.

A further related strand studies embedding dimension selection from a structural perspective, seeking principled criteria for identifying the minimal dimension needed to encode relevant information without sacrificing performance \citep{Gu2021c}. This perspective is closely related to classical resource allocation problems in operations research. For example, \citet{tenEikelder2023} study integer resource allocation problems with expensive function evaluations. \citet{Park2024} study a multi-period resource allocation problem with heterogeneous and non-shareable resources. 

A particularly relevant connection is with the optimal computing budget allocation (OCBA) literature. A limited simulation budget is distributed across competing alternatives to improve selection efficiency and reduce simulation cost, as in the foundational work on simulation budget allocation for ordinal optimization \citep{Chen2000a} and its later systematic treatment in stochastic simulation optimization \citep{Chen2010a}.
In black-box optimization, Bayesian optimization methods for constrained mixed-integer multi-objective problems under limited evaluation budgets are developed \citep{Duro2023}. In the broader literature on multi-objective resource allocation, limited resources must be distributed across competing activities under uncertainty \citep{Pan2022}. Finally, in data-driven operational problems, the objective may need to be learned from historical data rather than specified analytically \citep{Yu2026}.

Our contribution differs from existing resource allocation and dimension reduction frameworks by treating embedding dimension selection as an explicit constrained allocation problem. We study how representational capacity should be distributed across multiple heterogeneous categorical predictors when the total embedding budget is fixed. In predictive modeling, this budget serves as a global capacity parameter that governs the tradeoff between representational fidelity and finite-sample overfitting risk.

Operationally, the budget constraint represents the fact that embedding capacity is limited. Larger embeddings require more parameters, increase model complexity and lead to overfitting which deteriorates the model's predictive capacity. Therefore, embedding dimensions must be allocated selectively so that representational gains are balanced against the statistical cost of estimation in finite samples.

To address this issue, we develop a formal analysis of the tradeoff between approximation and estimation error induced by embeddings. On the basis of this characterization, we derive a closed-form allocation rule under the proposed risk approximation and the global budget constraint. The resulting methodology provides a theoretically grounded alternative to purely heuristic dimension selection and it situates embedding design within a rigorous constrained-allocation framework. In this respect, the paper contributes to the interface between operations research, statistics, and machine learning (ML) by establishing a principled basis for the allocation of embedding capacity under finite-budget constraints in predictive modeling.

Theoretical analysis is complemented by simulations and a real-data healthcare application that assess both validity and practical relevance. The simulations verify the spectral proxy for approximation error, illustrate the finite sample approximation-estimation tradeoff, and show that under a fixed global budget the proposed allocation rule outperforms uniform and cardinality-based heuristics, especially when predictor heterogeneity is substantial. The real-data application is included as a realistic healthcare prediction setting in which limited embedding capacity must be allocated across heterogeneous categorical variables. 

The remainder of the paper is organized as follows: Section~\ref{sec:method} introduces the modeling framework and formulates the embedding allocation problem. Section~\ref{sec:results} provides simulation evidence for the spectral characterization of approximation error, the finite sample approximation-estimation tradeoff, and the behavior of the proposed allocation rule under a fixed embedding budget. It also compares the proposed method with benchmark allocation strategies and reports a real-data experiment on a healthcare prediction task. Section~\ref{sec:concl} provides a summary of the main results, a discussion of their implications and directions for further work.

\section{Methodology}\label{sec:method}

This section develops the theoretical foundation of the proposed embedding-based framework. We first study the identification of the representation and then characterize the approximation error induced by finite-dimensional embeddings. We next turn to the estimation procedure and establish its convergence properties. Finally, we derive a generalization bound and show how the resulting approximation-estimation tradeoff yields an allocation rule for embedding dimensions across categorical predictors under a finite resource budget.

Throughout, let the predictive function be defined as
\[
f(\mathbf{x};\Theta)=\mathbf{w}^\top \boldsymbol{\phi}(\mathbf{x})+b,
\]
where \(\mathbf{x}\) denotes the input vector and \(b\in\mathbb{R}\) is an intercept term. Let \(c_j(\mathbf{x})\in\{1,\ldots,N_j\}\) denote the observed category of input \(\mathbf{x}\) in predictor \(j\). The feature representation is given by
\[
\boldsymbol{\phi}(\mathbf{x})
=
\Bigl[
\mathbf{z}(\mathbf{x})^\top,\;
\mathbf{u}_1(c_1(\mathbf{x}))^\top,\;
\ldots,\;
\mathbf{u}_m(c_m(\mathbf{x}))^\top
\Bigr]^\top,
\]
where \(\mathbf{z}(\mathbf{x})\in\mathbb{R}^p\) denotes the numerical covariates and \(\mathbf{u}_j(c_j(\mathbf{x}))\in\mathbb{R}^{r_j}\) denotes the representation vector associated with the \(j\)-th categorical predictor. For each predictor \(j=1,\ldots,m\), let
\[
\mathbf U^{(j)}\in \mathbb R^{N_j\times r_j},
\]
denote the representation matrix, where \(N_j\) is the number of categories in predictor \(j\) and \(r_j\) is the corresponding representation dimension. The vector \(\mathbf u_j(c)\) denotes the embedding vector associated with category \(c\), that is, the transpose of the \(c\)-th row of \(\mathbf U^{(j)}\). Hence, for an input \(\mathbf x\), the embedding used for predictor \(j\) is \(\mathbf u_j(c_j(\mathbf x))\).

 We partition the coefficient vector conformably as
\[
\mathbf{w}=
\Bigl[\mathbf{w}_z^\top,\;\mathbf{w}_1^\top,\;\ldots,\;\mathbf{w}_m^\top\Bigr]^\top,
\]
where \(\mathbf{w}_z\in\mathbb{R}^p\) and \(\mathbf{w}_j\in\mathbb{R}^{r_j}\) for \(j=1,\ldots,m\). Hence,
\[
f(\mathbf{x};\Theta)
=
\mathbf{w}_z^\top \mathbf{z}(\mathbf{x})
+\sum_{j=1}^m \mathbf{w}_j^\top \mathbf{u}_j(c_j(\mathbf{x}))
+b.
\]

We note that in the empirical analysis, the population representation matrix \(\mathbf{U}^{(j)}\) is not observed. Instead, a finite-dimensional learned representation
\(\widehat{\mathbf{U}}_{d_j}^{(j)}\in\mathbb{R}^{N_j\times d_j}\) is estimated from the training sample.

\subsection{Identification of the embedding representation}

We begin by clarifying a basic structural property of the model. Although the predictive function is uniquely determined by the joint effect of the embeddings and the prediction layer, the numerical coordinates of the learned embeddings themselves need not be unique. This issue matters for interpretation, because it determines which aspects of the learned representation are substantively meaningful. The next proposition shows that the predictive structure is invariant to orthogonal transformations of the embedding coordinates.

\begin{proposition}
\label{prop:nonidentifiability}
For each predictor \(j=1,\ldots,m\), let \(\mathbf{Q}_j\in\mathbb{R}^{r_j\times r_j}\) be any orthogonal matrix satisfying
\[
\mathbf{Q}_j^\top \mathbf{Q}_j=\mathbf{Q}_j\mathbf{Q}_j^\top=\mathbf{I}_{r_j}.
\]
Define the transformed representation matrix and coefficient vector by
\[
\widetilde{\mathbf{U}}^{(j)}=\mathbf{U}^{(j)}\mathbf{Q}_j,
\qquad
\widetilde{\mathbf{w}}_j=\mathbf{Q}_j^\top \mathbf{w}_j,
\]
while keeping \(\mathbf{w}_z\) and \(b\) unchanged. Then, for every input \(\mathbf{x}\),
\[
\widetilde{f}(\mathbf{x};\widetilde{\Theta})=f(\mathbf{x};\Theta).
\]
Hence, the representation is not uniquely identifiable. In particular, it is invariant under orthogonal transformations of the embedding coordinates.
\end{proposition}

\begin{proof}
For each predictor \(j\), the transformed representation associated with category \(c_j(\mathbf{x})\) is induced by replacing \(\mathbf{U}^{(j)}\) with
\[
\widetilde{\mathbf{U}}^{(j)}=\mathbf{U}^{(j)}\mathbf{Q}_j.
\]
Accordingly, the transformed representation vector is
\[
\widetilde{\mathbf{u}}_j(c_j(\mathbf{x}))=\mathbf{Q}_j^\top \mathbf{u}_j(c_j(\mathbf{x})),
\]
under the column vector convention used in the predictive function. Therefore,
\[
\widetilde{\mathbf{w}}_j^\top \widetilde{\mathbf{u}}_j(c_j(\mathbf{x}))
=
(\mathbf{Q}_j^\top \mathbf{w}_j)^\top (\mathbf{Q}_j^\top \mathbf{u}_j(c_j(\mathbf{x})))
=
\mathbf{w}_j^\top \mathbf{Q}_j \mathbf{Q}_j^\top \mathbf{u}_j(c_j(\mathbf{x}))
=
\mathbf{w}_j^\top \mathbf{u}_j(c_j(\mathbf{x})),
\]
where the last equality follows from orthogonality. Since \(\mathbf{w}_z\) and \(b\) remain unchanged, we obtain
\[
\widetilde{f}(\mathbf{x};\widetilde{\Theta})
=
\mathbf{w}_z^\top \mathbf{z}(\mathbf{x})+\sum_{j=1}^m \widetilde{\mathbf{w}}_j^\top \widetilde{\mathbf{u}}_j(c_j(\mathbf{x}))+b
=
\mathbf{w}_z^\top \mathbf{z}(\mathbf{x})+\sum_{j=1}^m \mathbf{w}_j^\top \mathbf{u}_j(c_j(\mathbf{x}))+b
=
f(\mathbf{x};\Theta).
\]
\end{proof}

The restriction to orthogonal transformations is mainly motivated by the geometric quantities used in the subsequent analysis. More general invertible linear reparameterizations can also leave the predictive function unchanged when the corresponding prediction-layer coefficients are transformed appropriately. However, arbitrary invertible transformations need not preserve the Euclidean geometry of the embedding space. By contrast, orthogonal transformations preserve inner products, Frobenius norms, and singular values, which are central to the spectral approximation argument developed below. Thus, the orthogonal case captures the relevant non-identifiability of the representation while keeping the geometry of the embedding space consistent with the approximation framework.

\begin{remark}
Proposition \ref{prop:nonidentifiability} shows that the numerical coordinates of the learned embeddings are not unique. Therefore, interpretation should focus on predictive performance and structural relationships rather than on the specific coordinate values of the embeddings. 
\end{remark}

\subsection{Approximation properties of finite-dimensional embeddings}

Having clarified the identification issue, we next study the approximation role of the embedding representation. In the proposed framework, each categorical predictor is represented through a latent category matrix. A finite embedding dimension corresponds to a low-rank reconstruction of this latent category structure. This representation improves tractability, but it may also induce approximation error whenever the underlying category structure is more complex than the chosen representation dimension permits.

In this section, approximation error refers to the representation-level compression error induced by replacing the full category-level latent representation with a lower-rank reconstruction in the same ambient latent space. The induced prediction-level perturbation is then controlled by this representation-level error through the prediction-layer coefficient vector. Since each observation uses the row of the latent matrix corresponding to its observed category, the relevant approximation error is induced by the row-wise reconstruction errors of the category vectors. The Frobenius residual provides a natural aggregate measure of these row-wise errors, and therefore connects the matrix approximation problem with the predictor-level approximation error studied below. The next proposition formalizes this effect.

\begin{proposition}
\label{prop:approximation}
Suppose that the true population predictor admits the representation
\[
f(\mathbf{x};\Theta)
=
\mathbf{w}_z^\top \mathbf{z}(\mathbf{x})
+\sum_{j=1}^m \mathbf{w}_j^\top \mathbf{u}_j(c_j(\mathbf{x}))
+b.
\]
where \(\mathbf{u}_j(c)\in\mathbb{R}^{r_j}\) is a latent category vector for category \(c\) in predictor \(j\). Define
\[
\mathbf{U}^{(j)}=
\begin{bmatrix}
\mathbf{u}_j(1)^\top\\
\vdots\\
\mathbf{u}_j(N_j)^\top
\end{bmatrix}
\in\mathbb{R}^{N_j\times r_j},
\]
and let
\[
\sigma_1^{(j)}\geq \sigma_2^{(j)}\geq \cdots
\]
denote its singular values.

For any matrix \(\mathbf{V}^{(j)}\in\mathbb{R}^{N_j\times r_j}\) with rank at most \(d_j\), let \(\mathbf{v}_j(c)\in\mathbb{R}^{r_j}\) denote the transpose of its \(c\)-th row. Thus, \(\mathbf{V}^{(j)}\) has the same ambient dimension as \(\mathbf{U}^{(j)}\), but rank at most \(d_j\). Equivalently, \(\mathbf{v}_j(c)\) is a rank-\(d_j\) reconstruction of the full latent vector \(\mathbf{u}_j(c)\) in the original latent space, rather than a \(d_j\)-dimensional vector. If \(\pi_{j,c}\) denotes the probability of category \(c\) for predictor \(j\), define the approximation error by
\[
\varepsilon_{\mathrm{approx}}^{(j)}(d_j)
=
\inf_{\operatorname{rank}(\mathbf{V}^{(j)})\leq d_j}
\left(
\sum_{c=1}^{N_j}
\pi_{j,c}
\left\|
\mathbf{u}_j(c)-\mathbf{v}_j(c)
\right\|_2^2
\right)^{1/2}.
\]
Then, if predictor \(j\) is represented using dimension \(d_j\),
\begin{equation}
\label{eq:1}
\varepsilon_{\mathrm{approx}}^{(j)}(d_j)
\leq
\left(
\sum_{k>d_j}\left(\sigma_k^{(j)}\right)^2
\right)^{1/2}.
\end{equation}
\end{proposition}

\begin{proof}
Let \(\mathbf{V}^{(j)}\in\mathbb{R}^{N_j\times r_j}\) be any matrix with rank at most \(d_j\), and let \(\mathbf{v}_j(c)\in\mathbb{R}^{r_j}\) denote the transpose of its \(c\)-th row. Since \(0\leq \pi_{j,c}\leq 1\), we have
\[
\left(
\sum_{c=1}^{N_j}
\pi_{j,c}
\left\|
\mathbf{u}_j(c)-\mathbf{v}_j(c)
\right\|_2^2
\right)^{1/2}
\leq
\left\|
\mathbf{U}^{(j)}-\mathbf{V}^{(j)}
\right\|_F.
\]
Taking the infimum over all matrices \(\mathbf{V}^{(j)}\) with rank at most \(d_j\) gives
\[
\varepsilon_{\mathrm{approx}}^{(j)}(d_j)
\leq
\inf_{\operatorname{rank}(\mathbf{V}^{(j)})\leq d_j}
\left\|
\mathbf{U}^{(j)}-\mathbf{V}^{(j)}
\right\|_F.
\]
By the Eckart--Young--Mirsky theorem \citep{Golub1987},
\[
\inf_{\operatorname{rank}(\mathbf{V}^{(j)})\leq d_j}
\left\|
\mathbf{U}^{(j)}-\mathbf{V}^{(j)}
\right\|_F^2
=
\sum_{k>d_j}\left(\sigma_k^{(j)}\right)^2.
\]
Therefore,
\[
\varepsilon_{\mathrm{approx}}^{(j)}(d_j)
\leq
\left(
\sum_{k>d_j}\left(\sigma_k^{(j)}\right)^2
\right)^{1/2}.
\]
\end{proof}

Proposition~\ref{prop:approximation} identifies the approximation component of the methodology. In particular, it shows that increasing the embedding dimension reduces the compression error incurred when category-specific latent structure is represented by a lower-rank reconstruction. Having established the approximation properties of the embedding representation, the next reasonable step is to study how these representations can be estimated from data. This motivates the sequential estimation analysis in Section~\ref{sec:seqestimation}, where we quantify how estimation error accumulates across the different stages of the procedure and how this interacts with the approximation bias characterized above.

\subsection{Sequential estimation}\label{sec:seqestimation}

We now turn from the population representation to its empirical estimation. Since the model contains both prediction layer parameters and estimated representation matrices, direct joint optimization may be computationally demanding. A natural alternative is to use a sequential block-wise estimation procedure. The next theorem provides a convergence guarantee for that procedure and therefore justifies it as a stable computational strategy.

\begin{theorem}
\label{thm:convergence}
Consider the regularized empirical risk minimization problem
\[
\min_{\mathbf{w},\,b,\,\widehat{\mathbf{U}}}\;
L(\mathbf{w},b,\widehat{\mathbf{U}})
=
\frac{1}{n}\sum_{i=1}^n \ell\bigl(y_i, f(\mathbf{x}_i;\mathbf{w},b,\widehat{\mathbf{U}})\bigr)
+
\lambda_w \|\mathbf{w}\|_2^2
+
\sum_{j=1}^m \lambda_j \|\widehat{\mathbf{U}}^{(j)}\|_F^2,
\]
where \(\widehat{\mathbf{U}}=(\widehat{\mathbf{U}}^{(1)},\ldots,\widehat{\mathbf{U}}^{(m)})\), \(\|\mathbf{w}\|_2\) denotes the Euclidean norm of \(\mathbf{w}\), and \(\|\widehat{\mathbf{U}}^{(j)}\|_F\) denotes the Frobenius norm of \(\widehat{\mathbf{U}}^{(j)}\). Assume that \(\ell(\cdot,\cdot)\) is continuous and differentiable, that \(\lambda_w>0\) and \(\lambda_j>0\) for all \(j=1,\ldots,m\), and that \(L(\mathbf{w},b,\widehat{\mathbf{U}})\) is bounded below.

Let \(\{(\mathbf{w}^t,b^t,\widehat{\mathbf{U}}^t)\}_{t\geq 0}\) be the sequence generated by the alternating scheme
\[
(\mathbf{w}^{t+1},b^{t+1})\in\arg\min_{\mathbf{w},b} L(\mathbf{w},b,\widehat{\mathbf{U}}^t),
\qquad
\widehat{\mathbf{U}}^{t+1}\in\arg\min_{\widehat{\mathbf{U}}} L(\mathbf{w}^{t+1},b^{t+1},\widehat{\mathbf{U}}).
\]
Then the sequence of objective values \(\{L(\mathbf{w}^t,b^t,\widehat{\mathbf{U}}^t)\}_{t\geq 0}\) is non-increasing. Under the standard regularity conditions for exact block coordinate descent, every accumulation point of the generated sequence is a stationary point of \(L\).
\end{theorem}

\begin{proof}
By construction, the update of \((\mathbf{w},b)\) at iteration \(t\) satisfies
\[
L(\mathbf{w}^{t+1},b^{t+1},\widehat{\mathbf{U}}^t)\leq L(\mathbf{w}^t,b^t,\widehat{\mathbf{U}}^t).
\]
Similarly, the update of \(\widehat{\mathbf{U}}\) satisfies
\[
L(\mathbf{w}^{t+1},b^{t+1},\widehat{\mathbf{U}}^{t+1})\leq L(\mathbf{w}^{t+1},b^{t+1},\widehat{\mathbf{U}}^t).
\]
Combining the two inequalities yields
\[
L(\mathbf{w}^{t+1},b^{t+1},\widehat{\mathbf{U}}^{t+1})\leq L(\mathbf{w}^t,b^t,\widehat{\mathbf{U}}^t),
\]
so the sequence of objective values is non-increasing.

Because \(L\) is bounded below, the sequence \(\{L(\mathbf{w}^t,b^t,\widehat{\mathbf{U}}^t)\}\) converges to a finite limit. Let \((\mathbf{w}^\bullet,b^\bullet,\widehat{\mathbf{U}}^\bullet)\) be an arbitrary accumulation point, and let \(\{(\mathbf{w}^{t_k},b^{t_k},\widehat{\mathbf{U}}^{t_k})\}_{k\geq 1}\) be a convergent subsequence with limit \((\mathbf{w}^\bullet,b^\bullet,\widehat{\mathbf{U}}^\bullet)\).

Since each block update is an exact minimization over the corresponding block, the generated sequence is an exact block coordinate descent sequence. Under the standard regularity conditions for exact block coordinate descent stated in the theorem, every accumulation point of such a sequence is coordinate-wise stationary. Hence, for any accumulation point \((\mathbf{w}^\bullet,b^\bullet,\widehat{\mathbf{U}}^\bullet)\), we have
\[
\nabla_{\mathbf{w}} L(\mathbf{w}^\bullet,b^\bullet,\widehat{\mathbf{U}}^\bullet)=0,
\qquad
\nabla_b L(\mathbf{w}^\bullet,b^\bullet,\widehat{\mathbf{U}}^\bullet)=0,
\qquad
\nabla_{\widehat{\mathbf{U}}} L(\mathbf{w}^\bullet,b^\bullet,\widehat{\mathbf{U}}^\bullet)=0.
\]
Therefore, \((\mathbf{w}^\bullet,b^\bullet,\widehat{\mathbf{U}}^\bullet)\) is a stationary point of \(L\).
\end{proof}

Theorem \ref{thm:convergence} establishes the optimization properties of the proposed estimation procedure by guaranteeing monotone descent of the objective and stationarity of accumulation points. With this algorithmic justification in place, the next section shifts attention from optimization to statistical performance, focusing on the generalization properties of the learned estimator.

\subsection{Generalization performance}

While the preceding result establishes the optimization behavior of the proposed estimation procedure, a full justification also requires a statistical characterization of out-of-sample performance. The estimated representation layer increases the expressive capacity of the model, but this additional flexibility comes at a cost in terms of statistical accuracy. The next theorem formalizes this tradeoff.

\begin{theorem}
\label{thm:generalization}
Assume that \(\|\mathbf{w}\|_2\leq B_w\), that \(|b|\leq B_b\), and that
\(\|\widehat{\mathbf{U}}^{(j)}\|_F\leq B_j\) for all \(j=1,\ldots,m\). Assume also that the numerical covariates are uniformly bounded, namely
\(\|\mathbf{z}(\mathbf{x})\|_2\leq B_z\) for all \(\mathbf{x}\). Let the loss function
\(\ell:\mathcal{Y}\times\mathbb{R}\to[0,1]\) be Lipschitz continuous in its second argument with constant \(L_\ell\). Let
\[
R(f)=\mathbb{E}_{(\mathbf{X},Y)}[\ell(Y,f(\mathbf{X}))],
\]
denote the expected risk, and let
\[
\widehat{R}(f)=\frac{1}{n}\sum_{i=1}^n \ell\bigl(y_i,f(\mathbf{x}_i)\bigr),
\]
denote the empirical risk. Then, with probability at least \(1-\delta\) over the draw of the
training sample, for  
\[
\mathcal{F}
=
\left\{
\mathbf{x}\mapsto \mathbf{w}^\top \boldsymbol{\phi}(\mathbf{x})+b
:\;
\|\mathbf{w}\|_2\leq B_w,\;
|b|\leq B_b,\;
\|\widehat{\mathbf{U}}^{(j)}\|_F\leq B_j,\;
j=1,\ldots,m
\right\},
\]
every \(f\in\mathcal{F}\) satisfies

\begin{equation}\label{eq2}
R(f)
\leq
\widehat{R}(f)
+
C\sqrt{\frac{\sum_{j=1}^m N_j d_j}{n}}
+
\sqrt{\frac{8\log(2/\delta)}{n}},
\end{equation}
where \(C>0\) is a constant depending on \(L_\ell\), \(B_w\), \(B_b\), \(B_z\), the fixed numerical dimension \(p\) of $\mathbf{z}(\mathbf{x})$, and the bounds \(B_1,\ldots,B_m\) on the estimated representation matrices.
\end{theorem}

\begin{proof}
For each \(f\in\mathcal{F}\), define the shifted loss
\[
\widetilde{\ell}_f(\mathbf{x},y)
=
\ell(y,f(\mathbf{x}))-\ell(y,0).
\]
By a standard generalization bound for Lipschitz losses \citep{bartlett2002rademacher}, it follows that, with probability at least \(1-\delta\), every \(f\in\mathcal{F}\) satisfies
\[
R(f)
\leq
\widehat{R}(f)
+
\widehat{\mathfrak{R}}_n(\widetilde{\ell}\circ\mathcal{F})
+
\sqrt{\frac{8\log(2/\delta)}{n}},
\]
where
\[
\widetilde{\ell}\circ\mathcal{F}
=
\bigl\{
(\mathbf{x},y)\mapsto \ell(y,f(\mathbf{x}))-\ell(y,0): f\in\mathcal{F}
\bigr\},
\]
and \(\widehat{\mathfrak{R}}_n(\widetilde{\ell}\circ\mathcal{F})\) denotes its empirical Rademacher complexity \citep{Lei2016}.

Now, for each fixed \(y\), the map
\[
u\mapsto \ell(y,u)-\ell(y,0),
\]
is \(L_\ell\)-Lipschitz and vanishes at the origin. Therefore, by the contraction property for Rademacher complexity, we obtain
\[
\widehat{\mathfrak{R}}_n(\widetilde{\ell}\circ\mathcal{F})
\leq
2L_\ell\,\widehat{\mathfrak{R}}_n(\mathcal{F}).
\]
Thus,
\[
R(f)
\leq
\widehat{R}(f)
+
2L_\ell\,\widehat{\mathfrak{R}}_n(\mathcal{F})
+
\sqrt{\frac{8\log(2/\delta)}{n}}.
\]

It therefore remains to bound \(\widehat{\mathfrak{R}}_n(\mathcal{F})\). By definition,
\[
\widehat{\mathfrak{R}}_n(\mathcal{F})
=
\mathbb{E}_{\boldsymbol{\sigma}}
\left[
\sup_{f\in\mathcal{F}}
\frac{2}{n}\sum_{i=1}^n \sigma_i f(\mathbf{x}_i)
\right],
\]
where \(\sigma_1,\ldots,\sigma_n\) are independent Rademacher random variables \citep{book2011}. Since
\[
f(\mathbf{x}_i)=\mathbf{w}^\top\boldsymbol{\phi}(\mathbf{x}_i)+b,
\]
we have
\[
\widehat{\mathfrak{R}}_n(\mathcal{F})
=
\mathbb{E}_{\boldsymbol{\sigma}}
\left[
\sup_{\|\mathbf{w}\|_2\leq B_w,\; |b|\leq B_b,\;\|\widehat{\mathbf{U}}^{(j)}\|_F\leq B_j}
\left\{
\mathbf{w}^\top
\left(
\frac{2}{n}\sum_{i=1}^n \sigma_i \boldsymbol{\phi}(\mathbf{x}_i)
\right)
+
b\left(\frac{2}{n}\sum_{i=1}^n \sigma_i\right)
\right\}
\right].
\]
Using the bound \(|b|\leq B_b\), this gives
\[
\widehat{\mathfrak{R}}_n(\mathcal{F})
\leq
\mathbb{E}_{\boldsymbol{\sigma}}
\left[
\sup_{\|\mathbf{w}\|_2\leq B_w,\;\|\widehat{\mathbf{U}}^{(j)}\|_F\leq B_j}
\mathbf{w}^\top
\left(
\frac{2}{n}\sum_{i=1}^n \sigma_i \boldsymbol{\phi}(\mathbf{x}_i)
\right)
\right]
+
B_b\,
\mathbb{E}_{\boldsymbol{\sigma}}
\left|
\frac{2}{n}\sum_{i=1}^n \sigma_i
\right|.
\]
Since
\[
\mathbb{E}_{\boldsymbol{\sigma}}
\left|
\frac{2}{n}\sum_{i=1}^n \sigma_i
\right|
\leq
\frac{2}{\sqrt n},
\]
the intercept contributes only an \(O(n^{-1/2})\) term. Applying the Cauchy--Schwarz inequality to the first term yields
\[
\widehat{\mathfrak{R}}_n(\mathcal{F})
\leq
B_w\,
\mathbb{E}_{\boldsymbol{\sigma}}
\left[
\sup_{\|\widehat{\mathbf{U}}^{(j)}\|_F\leq B_j}
\left\|
\frac{2}{n}\sum_{i=1}^n \sigma_i \boldsymbol{\phi}(\mathbf{x}_i)
\right\|_2
\right]
+
\frac{2B_b}{\sqrt n}.
\]

The feature map contains the fixed numerical part \(\mathbf{z}(\mathbf{x})\) and the embedding-dependent part. Since \(\|\mathbf{z}(\mathbf{x})\|_2\leq B_z\) and the numerical dimension \(p\) is fixed, the contribution of the numerical covariates is also of order \(O(n^{-1/2})\) and can be absorbed into the constant. The remaining dimension-dependent contribution comes from the embedding parameters.

It remains to relate this term to the number of free embedding parameters. For predictor \(j\), the learned embedding matrix \(\widehat{\mathbf{U}}^{(j)}\) has size \(N_j\times d_j\), and therefore contributes \(N_jd_j\) trainable parameters. Hence, the embedding-dependent part of the model is parameterized by $
D_{\mathrm{emb}}=\sum_{j=1}^m N_jd_j
$
free embedding parameters. Under the Frobenius norm constraints \(\|\widehat{\mathbf{U}}^{(j)}\|_F\leq B_j\), these parameters lie in a bounded subset of a \(D_{\mathrm{emb}}\)-dimensional Euclidean space. Moreover, the constraints \(\|\mathbf{w}\|_2\leq B_w\), \(|b|\leq B_b\), and \(\|\mathbf{z}(\mathbf{x})\|_2\leq B_z\) imply that the induced finite-dimensional parametric class is uniformly bounded and Lipschitz, with constants independent of \(D_{\mathrm{emb}}\). Therefore, applying the standard Rademacher complexity bound for bounded Lipschitz finite-dimensional parametric classes gives
\[
\widehat{\mathfrak{R}}_n(\mathcal{F})
\leq
C_1
\sqrt{
\frac{
\sum_{j=1}^m N_j d_j
}{n}
}
,
\]
for some constant \(C_1>0\), where \(C_1\) absorbs the bounded intercept term, the fixed numerical covariate contribution, and the norm bounds on the representation matrices.

Substituting this bound into the preceding inequality gives
\[
R(f)
\leq
\widehat{R}(f)
+
2L_\ell C_1 
\sqrt{
\frac{
\sum_{j=1}^m N_j d_j
}{n}
}
+
\sqrt{\frac{8\log(2/\delta)}{n}}.
\]
Setting \(C=2L_\ell C_1\) yields Equation~\eqref{eq2}.
\end{proof}

\subsection{Embedding-dimension allocation under a resource constraint}

Together with Proposition~\ref{prop:approximation}, Theorem~\ref{thm:generalization} formalizes the approximation--estimation tradeoff induced by the embedding dimensions. Larger dimensions can reduce approximation error by capturing more latent category structure, but they also increase the number of learned embedding parameters and hence the statistical complexity of the function class. Thus, overly large embedding dimensions may lead to overfitting and weaker out-of-sample performance in finite samples. This tradeoff naturally motivates the problem of how to allocate a fixed embedding budget across categorical predictors. 

The allocation problem considered here is closely related in spirit to the OCBA literature. In OCBA, a limited computing or simulation budget is allocated across competing alternatives in order to improve selection efficiency, typically by increasing the probability of correct selection. Similarly, our framework allocates a fixed budget across categorical predictors that compete for limited representational capacity. The key difference is that the allocated resource is not simulation replications or computing time, but embedding capacity. Thus, the proposed formulation can be viewed as an embedding-capacity analogue of budget allocation ideas from stochastic simulation optimization.
The following theorem derives a closed-form allocation rule for an approximation-based surrogate problem under a fixed embedding budget.

\begin{theorem}
\label{thm:allocation}
Suppose that the following surrogate risk criterion is used to approximate the approximation--estimation tradeoff
\[
R(f)
=
R^\star
+
\sum_{j=1}^m \varepsilon_{\mathrm{approx}}^{(j)}(d_j)
+
C\sqrt{\frac{\sum_{j=1}^m N_j d_j}{n}},
\]
where \(R^\star\) denotes the irreducible risk, \(\varepsilon_{\mathrm{approx}}^{(j)}(d_j)\) is the approximation error associated with predictor \(j\), and the last term represents the estimation error. To obtain an explicit allocation rule, we adopt the working model
\[
\varepsilon_{\mathrm{approx}}^{(j)}(d_j)=\frac{a_j}{d_j},
\qquad
a_j>0,
\]
where \(a_j\) is a predictor-specific approximation parameter. 

Consider the corresponding approximation-based surrogate allocation problem
\[
\min_{d_1,\ldots,d_m>0}
\sum_{j=1}^m \frac{a_j}{d_j}
\qquad
\text{subject to}
\qquad
\sum_{j=1}^m N_j d_j \leq B,
\]
where \(N_j\) is the number of categories in predictor \(j\) and \(B\) is a fixed embedding budget. Then the solution of this surrogate allocation problem is given by
\[
d_j^\star \propto \sqrt{\frac{a_j}{N_j}},
\]
and the proportionality constant is determined by
\[
\sum_{j=1}^m N_j d_j^\star = B.
\]
In particular,
\[
d_j^\star
=
\frac{B\,\sqrt{a_j/N_j}}{\sum_{k=1}^m \sqrt{a_k N_k}},
\qquad j=1,\ldots,m.
\]
\end{theorem}

\begin{proof}
The risk decomposition contains both approximation and estimation terms. Under the fixed budget constraint $\sum_{j=1}^m N_jd_j\leq B$, the estimation term $\sqrt{\sum_{j=1}^m N_j d_j/n}$ is bounded above by $\sqrt{B/n}$. The surrogate allocation problem then focuses on the approximation terms, which determine how the fixed representational capacity should be distributed across predictors. We therefore consider the following approximation-based surrogate allocation problem
\[
\min_{d_1,\ldots,d_m>0}\;
\sum_{j=1}^m \frac{a_j}{d_j}
\qquad
\text{subject to}
\qquad
\sum_{j=1}^m N_j d_j \leq B.
\]
Since the objective is decreasing in each \(d_j\), the budget constraint is active at the solution. Introduce a Lagrange multiplier \(\lambda>0\) and define
\[
\mathcal{L}(\mathbf{d},\lambda)
=
\sum_{j=1}^m \frac{a_j}{d_j}
+
\lambda\left(\sum_{j=1}^m N_j d_j-B\right),
\]
where \(\mathbf{d}=(d_1,\ldots,d_m)^\top\). The first-order conditions are
\[
\frac{\partial \mathcal{L}}{\partial d_j}
=
-\frac{a_j}{d_j^2}+\lambda N_j
=
0,
\qquad j=1,\ldots,m.
\]
Solving for \(d_j\) yields
\[
d_j=\sqrt{\frac{a_j}{\lambda N_j}}.
\]
Hence,
\[
d_j^\star \propto \sqrt{\frac{a_j}{N_j}}.
\]
The proportionality constant is determined by enforcing the active constraint
\[
\sum_{j=1}^m N_j d_j^\star = B.
\]
This yields
\[
d_j^\star
=
\frac{B\,\sqrt{a_j/N_j}}{\sum_{k=1}^m \sqrt{a_k N_k}},
\qquad j=1,\ldots,m.
\]
\end{proof}

The allocation rule should be interpreted as the solution of an approximation-based surrogate problem, rather than as the exact optimizer of the full finite sample risk. Also, the estimation term is not removed from the risk decomposition. Under a fixed budget, it is controlled by the total embedding parameter budget, so the relative allocation across predictors is driven by approximation differences. Finally, the budget \(B\) can be interpreted as a global capacity parameter governing the tradeoff between representation quality and overfitting.

In implementation, the approximation error for predictor \(j\) at dimension \(d_j\) is estimated from the empirical singular-value tail of its learned embedding matrix,
\[
\widehat{\varepsilon}_{\mathrm{approx}}^{(j)}(d_j)\propto
\left(\sum_{k>d_j}\bigl(\widehat{\sigma}_k^{(j)}\bigr)^2\right)^{1/2},
\]
and the coefficient \(\widehat a_j\) is obtained by fitting
\[
\widehat{\varepsilon}_{\mathrm{approx}}^{(j)}(d_j)\approx \frac{\widehat a_j}{d_j}.
\]
Thus, the working model \(\varepsilon_{\mathrm{approx}}^{(j)}(d_j)=a_j/d_j\) should be viewed as a tractable one-parameter summary of the empirical approximation profile, not as an exact description of all possible spectral decay patterns.

These considerations motivate the practical procedures used in the simulation study. Algorithm~\ref{alg:embedding_allocation} provides the implementable fixed-budget allocation rule associated with Theorem~\ref{thm:allocation}, with approximation coefficients estimated from empirical spectral decay as suggested by Proposition~\ref{prop:approximation}. Since the total embedding budget is not known a priori, Algorithm~\ref{alg:budget_selection} adds an outer selection step over a prespecified candidate set \( \mathcal{B} \). For each candidate budget \( B \in \mathcal{B} \), Algorithm~\ref{alg:budget_selection} calls Algorithm~\ref{alg:embedding_allocation} to obtain \( \widehat{\mathbf d}\), trains the corresponding model, and evaluates its validation performance. The selected budget \( \widehat{B} \) is then combined with Algorithm~\ref{alg:embedding_allocation} on the full training sample to produce the final dimension vector \( \widehat{\mathbf d} \).

\begin{algorithm}[t]\scriptsize
\caption{Embedding-dimension allocation under a budget constraint}
\label{alg:embedding_allocation}
\textbf{Require:} Training data \(\{(\mathbf{x}_i,y_i)\}_{i=1}^n\), categorical predictors \(j=1,\ldots,m\), numbers of categories \(N_j\), total resource budget \(B\geq \sum_{j=1}^m N_j\), initial dimensions \(\bar d_j\), regularization parameters, and positivity constant \(\eta>0\).

\begin{enumerate}
    \item \textbf{Initial estimation.}
    For each categorical predictor \(j\), choose an initial dimension \(\bar d_j\) and estimate the embedding-based predictive model
    \[
    \widehat f(\mathbf{x})
    =
    \widehat{\mathbf w}^{\top}\widehat{\boldsymbol\phi}(\mathbf{x})+\widehat b,
    \qquad
    \widehat{\boldsymbol\phi}(\mathbf{x})
    =
    \Big[
    \mathbf z(\mathbf{x})^\top,\,
    \widehat{\mathbf u}_1(c_1(\mathbf{x}))^\top,\,
    \ldots,\,
    \widehat{\mathbf u}_m(c_m(\mathbf{x}))^\top
    \Big]^\top .
    \]
    
    \item \textbf{Estimation of approximation coefficients.}
    For each predictor \(j\), construct the learned embedding matrix
    \[
    \widehat{\mathbf{U}}_{\bar d_j}^{(j)}  \in \mathbb{R}^{N_j \times \bar d_j},
    \]
    and compute its singular values
    \[
    \widehat{\sigma}_1^{(j)} \geq \widehat{\sigma}_2^{(j)} \geq \cdots \geq \widehat{\sigma}_{\bar d_j}^{(j)}.
    \]
    Using the empirical spectral decay, estimate the approximation coefficient in the working model
    \[
    \widehat\varepsilon_{\mathrm{approx}}^{(j)}(d_j) \approx \frac{\widehat a_j}{d_j},
    \]
    and denote the resulting estimate by \(\widehat a_j\). Set
    \[
    \widehat a_j \leftarrow \max\{\widehat a_j,\eta\},
    \]
    for a small fixed constant \(\eta>0\), so that all estimated approximation coefficients are strictly positive.

    \item \textbf{Continuous dimension allocation.}
    Solve
    \[
    \min_{d_1,\ldots,d_m}
    \sum_{j=1}^m \frac{\widehat{a}_j}{d_j}
    \qquad
    \text{s.t.}
    \qquad
    \sum_{j=1}^m N_j d_j \leq B,
    \quad d_j>0.
    \]
    The solution is
    \[
    \widetilde{d}_j
    =
    \frac{B\sqrt{\widehat{a}_j/N_j}}
    {\sum_{k=1}^m \sqrt{\widehat{a}_k N_k}},
    \qquad j=1,\ldots,m.
    \]

    \item \textbf{Integer projection.}
    Set
    \[
    d_j'=\max\{1,\lfloor \widetilde{d}_j \rfloor\},
    \qquad j=1,\ldots,m.
    \]
    If the rounded allocation does not satisfy
    \[
    \sum_{j=1}^m N_j d_j' \leq B,
    \]
    project the rounded vector onto the feasible budget set. Any remaining budget is then allocated greedily to the predictors with the largest marginal reduction per unit resource cost
    \[
    \frac{\Delta_j(d_j')}{N_j}
    =
    \frac{1}{N_j}
    \left(
    \frac{\widehat a_j}{d_j'}
    -
    \frac{\widehat a_j}{d_j'+1}
    \right),
    \]
    until no further dimension can be added without violating the budget constraint.

    \item \textbf{Final estimation.}
    Re-estimate the predictive model
    \[
    \widehat f(\mathbf{x})
    =
    \widehat{\mathbf w}^{\top}\widehat{\boldsymbol\phi}(\mathbf{x})
    +
    \widehat b,
    \]
    using the selected dimensions \(d_1',\ldots,d_m'\).

    \item \textbf{Return} \(d_1',\ldots,d_m'\).
\end{enumerate}
\end{algorithm}

\begin{algorithm}[t]\scriptsize
\caption{Data-driven selection of the total embedding budget}
\label{alg:budget_selection}
\begin{algorithmic}
\Require Dataset \(\mathcal{D}\), candidate budget set \(\mathcal{B}\), number of inner folds \(K\), number of repetitions \(R\), allocation subroutine \(\mathcal{A}(\mathcal{D}_{\mathrm{tr}}, B)\), training procedure \(\mathcal{T}(\mathcal{D}_{\mathrm{tr}}, \mathbf{d})\), validation loss function \(L_{\mathrm{val}}(\cdot,\cdot)\).
\Ensure Selected total budget \(\widehat{B}\) and associated dimension vector \(\widehat{\mathbf{d}}\).

\For{each \(B \in \mathcal{B}\)}
    \State Initialize score list \(\mathcal{S}_B \leftarrow \emptyset\).
    \For{\(r = 1,\dots,R\)}
        \State Randomly partition \(\mathcal{D}\) into \(K\) folds.
        \For{\(k = 1,\dots,K\)}
            \State Define training split \(\mathcal{D}_{\mathrm{tr}}^{(r,k)}\) and validation split \(\mathcal{D}_{\mathrm{val}}^{(r,k)}\).
            \State Compute \(\widehat{\mathbf{d}}^{(r,k)}(B) \leftarrow \mathcal{A}(\mathcal{D}_{\mathrm{tr}}^{(r,k)}, B)\).
            \State Train model \(f^{(r,k)}_{B} \leftarrow \mathcal{T}(\mathcal{D}_{\mathrm{tr}}^{(r,k)},
            \widehat{\mathbf{d}}^{(r,k)}(B))\).
            \State Evaluate validation loss
            \[
            s^{(r,k)}_{B} \leftarrow L_{\mathrm{val}}\!\left(f^{(r,k)}_{B}, \mathcal{D}_{\mathrm{val}}^{(r,k)}\right).
            \]
            \State Append \(s^{(r,k)}_{B}\) to \(\mathcal{S}_B\).
        \EndFor
    \EndFor
    \State Compute mean validation loss
    \[
    \overline{L}(B) \leftarrow \frac{1}{|\mathcal{S}_B|} \sum_{s \in \mathcal{S}_B} s.
    \]
    \State Compute standard error
    \[
    \mathrm{SE}(B) \leftarrow \frac{\operatorname{sd}(\mathcal{S}_B)}{\sqrt{|\mathcal{S}_B|}}.
    \]
\EndFor

\State Let
\[
B_{\min} \in \arg\min_{B \in \mathcal{B}} \overline{L}(B).
\]

\State Apply the one-standard-error rule and choose
\[
\widehat{B} \leftarrow \min \left\{ B \in \mathcal{B} \,\middle|\, \overline{L}(B) \leq \overline{L}(B_{\min}) + \mathrm{SE}(B_{\min}) \right\}.
\]

\State Compute final allocation on the full dataset
\[
\widehat{\mathbf{d}} \leftarrow \mathcal{A}(\mathcal{D}, \widehat{B}).
\]

\State \Return \(\widehat{B}, \widehat{\mathbf{d}}\).
\end{algorithmic}
\end{algorithm}

\section{Simulation Study}\label{sec:results}

In this section, we present simulation evidence for the main theoretical results of Section~\ref{sec:method}. Specifically, the following three experiments examine the spectral characterization of approximation error, the finite-sample approximation--estimation tradeoff, and the performance of the proposed allocation rule under fixed embedding budget constraints.

Throughout, we consider categorical predictors with latent low-rank structure. For a single predictor with \(N\) categories, let \(\mathbf{U}\in\mathbb{R}^{N\times r}\) denote the latent category matrix, where the \(c\)-th row \(\mathbf{u}_c\in\mathbb{R}^r\) represents the latent embedding vector of category \(c\). We generate
\[
\mathbf{U}=\mathbf{Q}_1\boldsymbol{\Sigma}\mathbf{Q}_2^\top,
\]
where \(\boldsymbol{\Sigma}=\operatorname{diag}(\sigma_1,\ldots,\sigma_r)\) contains singular values following either polynomial decay, \(\sigma_k\propto k^{-\alpha}\), or exponential decay, \(\sigma_k\propto e^{-\alpha k}\).

The latent matrix \(\mathbf{U}\) is not directly observed. Instead, for each category \(c\in\{1,\ldots,N\}\), we generate noisy samples
\begin{equation}\label{eq3}
\mathbf{X}_{c,t}=\mathbf{u}_c+\boldsymbol{\varepsilon}_{c,t},
\qquad
\boldsymbol{\varepsilon}_{c,t}\sim\mathcal{N}(\mathbf{0},\sigma^2\mathbf{I}), \qquad c=1,...,N,
\end{equation}
and define the empirical embedding matrix \(\widehat{\mathbf U}\) from the corresponding category means
\begin{equation}\label{eq4}
\widehat{\mathbf u}_c=\frac{1}{n_c}\sum_{t=1}^{n_c}\mathbf{X}_{c,t}, \quad c=1,...,N.
\end{equation}

The category counts \(n_c\) are generated either uniformly or from a Dirichlet--multinomial distribution, thereby inducing heterogeneous estimation noise across categories. For a given embedding dimension \(d\), let \(\widehat{\mathbf U}^{(d)}\) denote the rank-\(d\) approximation of \(\widehat{\mathbf U}\) obtained by truncated singular value decomposition. 

\subsection{Approximation Error Under Finite-Sample Estimation}
\label{subsec:exp1}

This experiment evaluates the approximation error characterization in Proposition \ref{prop:approximation} under finite-sample conditions. The objective is to assess whether the singular-value spectrum of an estimated embedding matrix provides an accurate proxy for approximation error in a setting where the latent structure is not directly observed.

We consider a categorical predictor with $N$ categories and an underlying latent representation matrix $\mathbf{U} \in \mathbb{R}^{N \times r}$. The matrix $\mathbf{U}$ is constructed as $\mathbf{U} = \mathbf{Q}_1 \boldsymbol{\Sigma} \mathbf{Q}_2^\top$, where $\mathbf{Q}_1 \in \mathbb{R}^{N \times r}$ and $\mathbf{Q}_2 \in \mathbb{R}^{r \times r}$ are orthonormal matrices obtained from QR decompositions of Gaussian random matrices, and $\boldsymbol{\Sigma} = \mathrm{diag}(\sigma_1,\dots,\sigma_r)$ contains the singular values. The spectral design is specified to follow either polynomial or exponential decay, thereby inducing different levels of intrinsic dimensionality. For each category $c$, we generate independent observations using the form provided by Equation (\ref{eq3}), i.e.
\[
\mathbf{X}_{c,t} = \mathbf{u}_c + \boldsymbol{\varepsilon}_{c,t}, \quad \boldsymbol{\varepsilon}_{c,t} \sim \mathcal{N}(\mathbf{0}, \sigma^2 \mathbf{I}_r), \quad c=1,...,N,
\]
with $\sigma=1$, where $\mathbf{u}_c$ denotes the $c$-th row of $\mathbf{U}$. The number of observations per category is random and drawn from a Dirichlet-multinomial distribution. It is evident that for the purposes of this experiment we consider the existence of $1$ categorical predictor. The rows of the empirical embedding matrix $\widehat{\mathbf{U}}$ are computed according to the sample means shown in Equation (\ref{eq4}).

For each embedding dimension \(d\in\{1,\ldots,r\}\), we compute the rank-\(d\) approximation
\(\widehat{\mathbf U}^{(d)}\) of the population latent matrix \(\mathbf U\), and measure the 
approximation error using the quantity \(\|\mathbf U-\widehat{\mathbf U}^{(d)}\|_F\).
This is compared with the empirical spectral-tail proxy
\[
\left(\sum_{k>d}\widehat{\sigma}_k^2\right)^{1/2},
\]
where \(\widehat{\sigma}_k\) denotes the singular values of the estimated matrix
\(\widehat{\mathbf U}\). For each spectral design, we performed 150 Monte Carlo simulations and all reported quantities are averaged across simulations.

Figure~\ref{fig:exp1_main} presents the approximation error profiles across the considered spectral decay regimes. In all cases, the approximation error decreases with the embedding dimension, with the rate of decrease depending on the underlying singular-value decay. Under exponential decay, the error falls rapidly, whereas under polynomial decay the decline is more gradual, reflecting a more diffuse latent structure. Overall, Figure~\ref{fig:exp1_main} shows that the empirical singular-value tail follows the same qualitative pattern as the population approximation error profile. Figure~\ref{fig:exp1_gap} further plots the difference between the empirical spectral-tail proxy and the population approximation error. The gap remains small and decreases toward negligible values as the embedding dimension increases. Hence, despite finite-sample noise, the empirical spectrum of \(\widehat{\mathbf U}\) provides a useful proxy for approximation error.

Figure~\ref{fig:exp1_loglog} further supports the representation-level approximation interpretation used in Proposition~\ref{prop:approximation}. Across the three spectral regimes, the predictive error induced by rank-\(d\) representations decreases as the embedding dimension increases, following the same qualitative pattern as the population and empirical spectral tails. The results indicate that the spectral-tail quantity contains useful information about the predictive loss incurred when the latent category representation is compressed to rank \(d\).
The agreement should be interpreted as directional rather than exact. The spectral tail measures the total residual variation left outside the rank-\(d\) representation, whereas predictive error depends only on the components of this residual variation that are relevant for the response. Therefore, the predictive error may decrease more rapidly than the spectral tail when the main predictive signal is concentrated in the leading singular directions.

The right panels provide an additional diagnostic for the working approximation model used in Theorem~\ref{thm:allocation}. The fitted decay curves capture the main empirical approximation trends across the considered spectral regimes, although the fit is not exact in all cases. This supports the use of \(\widehat{\varepsilon}_{\mathrm{approx}}(d)\approx \widehat{a}/d\) as a tractable summary of the empirical approximation profile. The fitted coefficient \(\widehat{a}\) should therefore be interpreted as a data-driven allocation score, rather than as an exact population parameter or a universal spectral decay law.

\begin{figure}[t]
    \centering
    \includegraphics[width=1\textwidth]{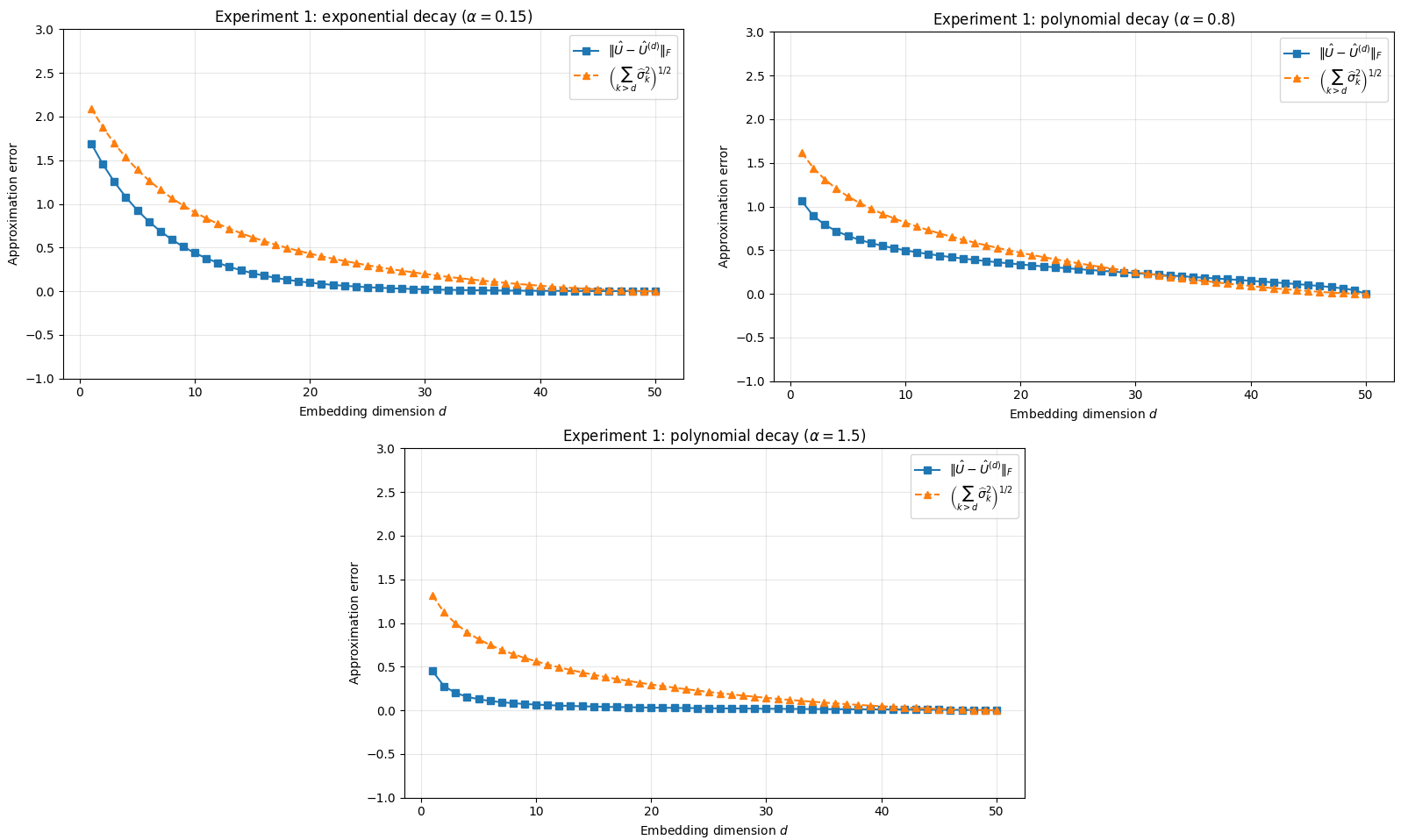}
    \caption{\small Approximation error as a function of embedding dimension $d$ under different spectral decay regimes.}
    \label{fig:exp1_main}
\end{figure}

\begin{figure}[t]
    \centering
    \includegraphics[width=1\textwidth]{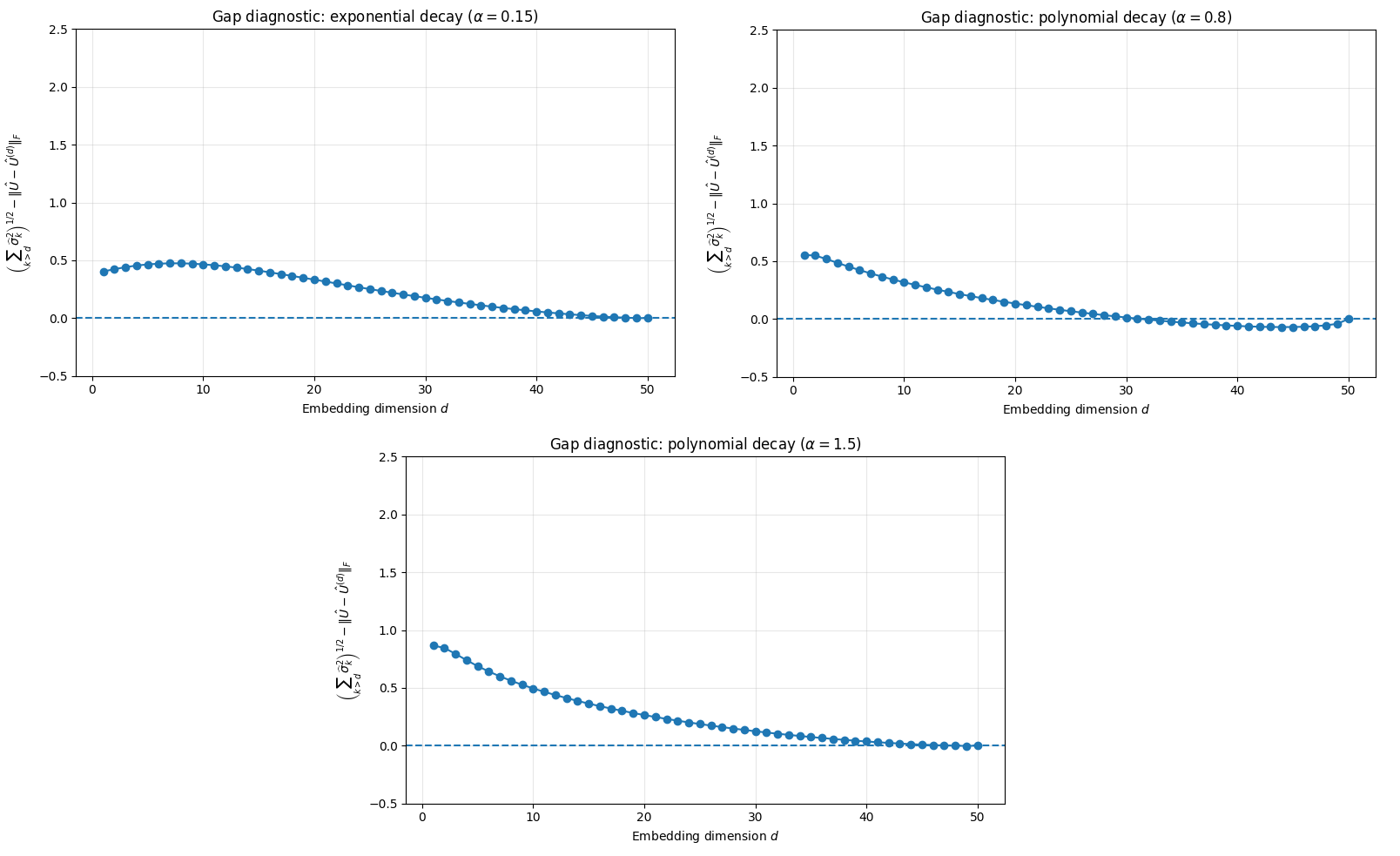}
    \caption{\small Gap diagnostic for Experiment 1. The figure plots the difference between the empirical spectral-tail proxy and the population approximation error across embedding dimensions.}
    \label{fig:exp1_gap}
\end{figure}

\begin{figure}[t]
    \centering
    \includegraphics[width=0.7\textwidth]{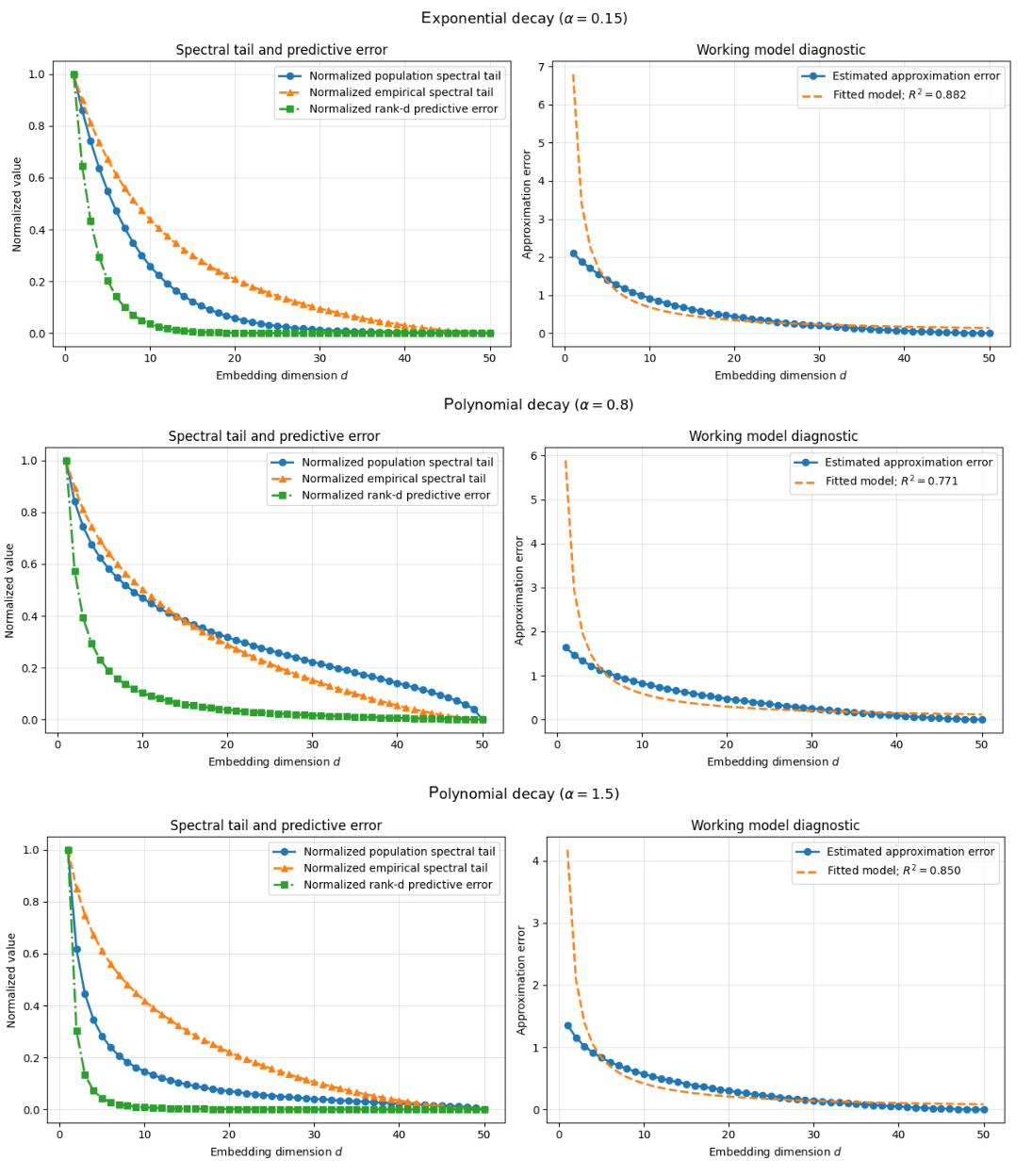}
    \caption{Predictive and approximation diagnostics under alternative spectral decay regimes. The left panels compare the population and empirical spectral tails with the predictive error induced by rank-\(d\) representations. The right panels compare the estimated approximation error with the fitted decay model \(\widehat{\varepsilon}_{\mathrm{approx}}(d)\approx \widehat a/d\).}
    \label{fig:exp1_loglog}
\end{figure}

\subsection{Approximation--Estimation tradeoff}
\label{subsec:exp2}

The second experiment examines the finite-sample tradeoff underlying Theorem \ref{thm:generalization}. The objective is to quantify how the choice of embedding dimension affects predictive performance when richer representations reduce approximation bias but simultaneously increase estimation error. To this end, we generate a latent matrix $\mathbf{U} \in \mathbb{R}^{N \times r}$ with controlled singular-value decay and construct noisy category-level observations from which an estimated embedding matrix is learned. In this experiment, we consider a single categorical predictor. For each category $c$, training observations are generated according to Equation (\ref{eq3}) with $\sigma=1$. An independent test sample is generated analogously. Category frequencies in the training sample are heterogeneous and drawn from a Dirichlet-multinomial model, so that estimation precision varies across categories. For each embedding dimension $d$, we compute the rank-$d$ approximation $\widehat{\mathbf{U}}^{(d)}$ of the estimated matrix $\widehat{\mathbf{U}}$ by truncated singular value decomposition and evaluate both in-sample and out-of-sample prediction error. The rows of  $\widehat{\mathbf{U}}$ are defined based on Equation (\ref{eq4}). Finally, we performed 150 Monte Carlo replications for each training sample size and all reported results are averaged across simulations.

Figure~\ref{fig:exp2_train_test} reports training and test mean squared error (MSE) as a function of the embedding dimension for three sample sizes. In all cases, training error declines monotonically with $d$, reflecting the increasing flexibility of the representation. Test error, however, exhibits the finite-sample tradeoff predicted by the theory. For the smallest sample size, the out-of-sample error reaches its minimum at a very low dimension and then increases steadily, indicating that estimation error rapidly dominates the gains from additional representation capacity. As the sample size increases, the location of the minimum shifts to the right and the deterioration beyond the optimum becomes less pronounced. This pattern is consistent with the theoretical prediction that the admissible embedding complexity expands with the amount of available data.

The underlying mechanism is made explicit in Figure~\ref{fig:exp2_optimal_d}, which decomposes the total reconstruction error into an approximation and an estimation error. The approximation error decreases with $d$, since larger embeddings capture a greater share of the latent structure. By contrast, the estimation error increases with $d$, reflecting the larger number of parameters that must be learned from finite samples. Their interaction produces the non-monotone profile of the total reconstruction error. The tradeoff is strongest in the smallest-sample regime and becomes flatter as the sample size grows, which explains the corresponding behavior of test performance in Figure~\ref{fig:exp2_train_test}.

Figure~\ref{fig:exp2_decomp} summarizes this pattern by plotting the empirically optimal embedding dimension against the training sample size. The selected dimension increases monotonically with $n$, providing evidence that larger samples support richer embedding representations. Taken together, the results provide strong empirical support for Theorem \ref{thm:generalization}. They show that embedding dimension should be chosen by balancing approximation gains against finite-sample estimation costs and that the optimal choice is inherently sample-size dependent.

\begin{figure}[t]
    \centering
    \includegraphics[width=1\textwidth]{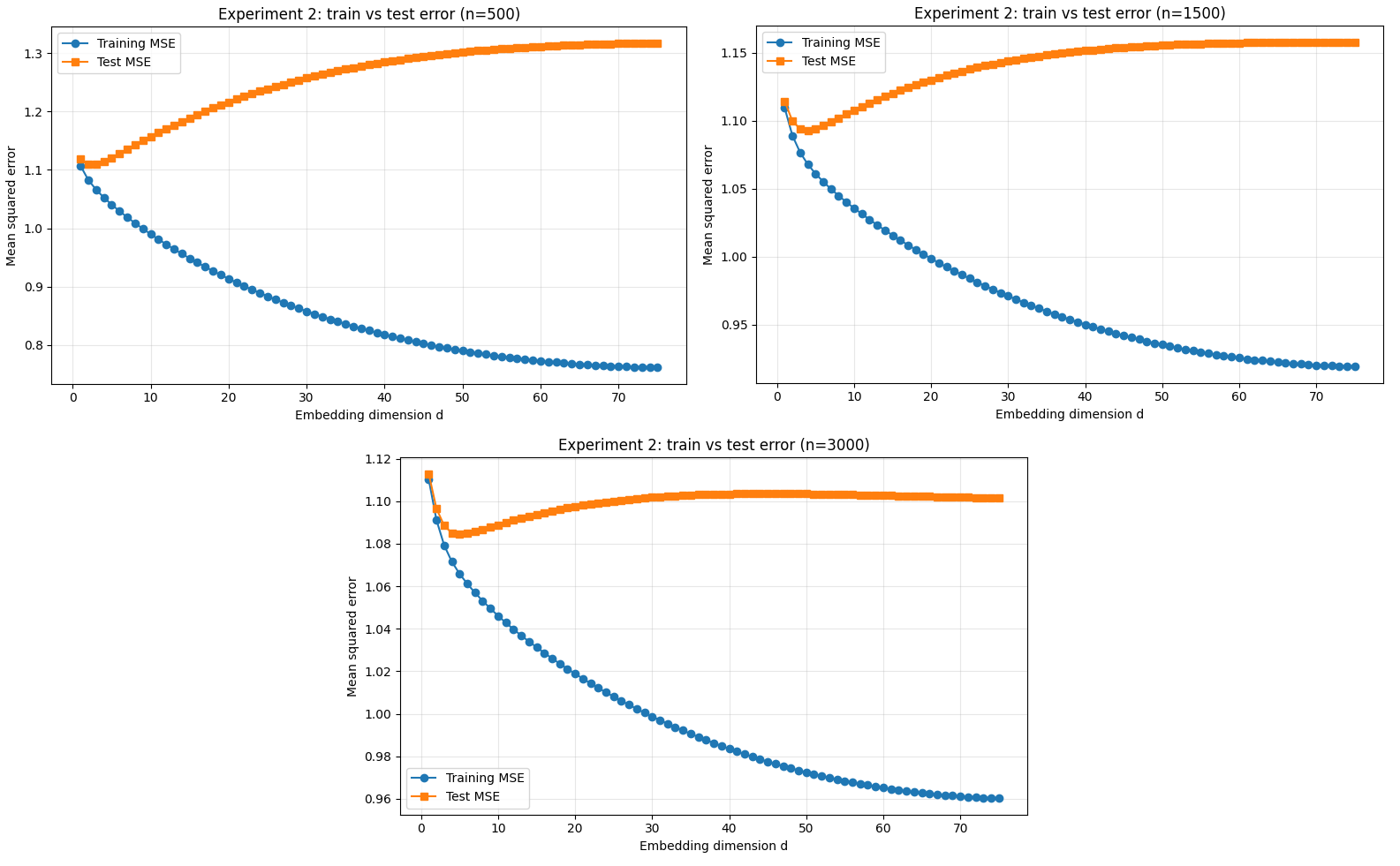}
    \caption{\small Training and test MSE as a function of embedding dimension for different training sample sizes.}
    \label{fig:exp2_train_test}
\end{figure}

\begin{figure}[t]
    \centering
    \includegraphics[width=1\textwidth]{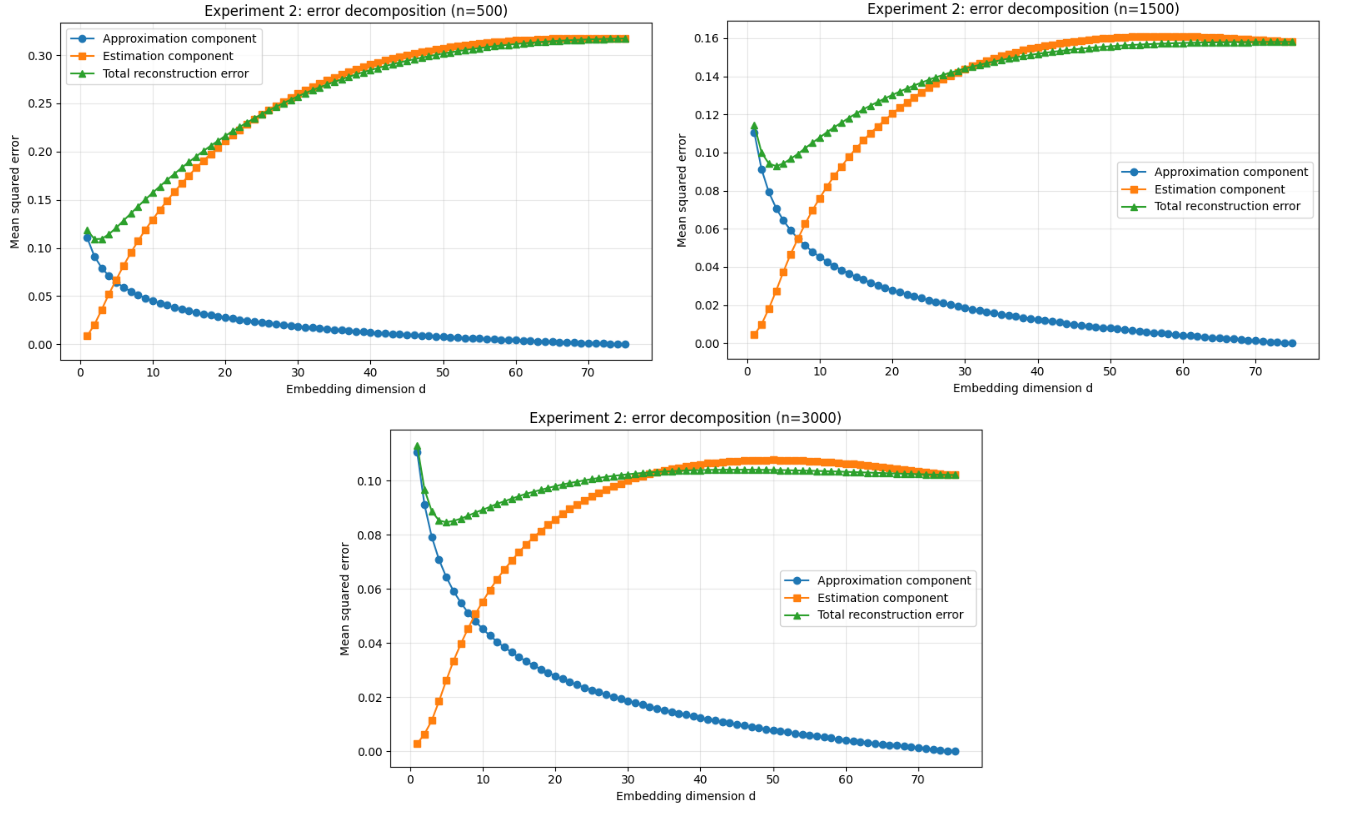}
    \caption{\small Decomposition of reconstruction error into approximation and estimation components for different training sample sizes.}
    \label{fig:exp2_optimal_d}
\end{figure}

\begin{figure}[t]
    \centering
    \includegraphics[width=0.7\textwidth]{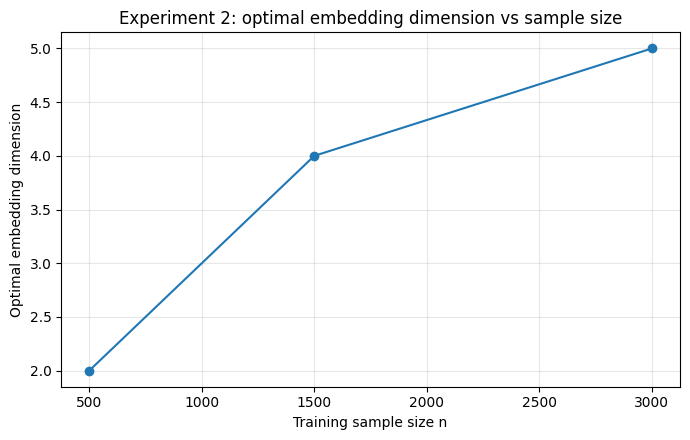}
    \caption{\small Empirically optimal embedding dimension as a function of training sample size.}
    \label{fig:exp2_decomp}
\end{figure}

\subsection{Evaluation of the Allocation Rule under a Fixed Embedding Budget}\label{subsec:exp3}

We evaluate the allocation rule developed in Section \ref{sec:method} in a multi-predictor setting under a fixed embedding-budget constraint. The objective is to quantify the effect of dimension allocation on out-of-sample performance and to assess whether the proposed law improves upon standard allocation heuristics in finite-sample regimes.

We consider \(m=3\) categorical predictors. For each predictor \(j\), let \(\mathbf{U}^{(j)} \in \mathbb{R}^{N_j \times r_j}\) denote the population latent matrix introduced in Section \ref{sec:method}. In the simulation design, \(\mathbf{U}^{(j)}\) is generated as
\[
\mathbf{U}^{(j)} = \mathbf{Q}_{1j}\,\mathrm{diag}(\sigma_{j1},\ldots,\sigma_{jr_j})\,\mathbf{Q}_{2j}^\top,
\]
where \(\mathbf{Q}_{1j} \in \mathbb{R}^{N_j \times r_j}\) has orthonormal columns, \(\mathbf{Q}_{2j} \in \mathbb{R}^{r_j \times r_j}\) is orthogonal, and the singular values \(\sigma_{jk}\) follow either polynomial or exponential decay, namely $
\sigma_{1k} \propto k^{-0.35}$, $\sigma_{2k} \propto k^{-1.5}$ and $\sigma_{3k} \propto e^{-0.3k}$. This construction induces heterogeneous spectral structures across predictors. Moreover, we set $N_1=30$, $N_2=90$ and $N_3=180$.

For each category \(c \in \{1,\ldots,N_j\}\), we observe noisy samples of the latent vector
\[
\mathbf{X}_{j,c,\ell} = \mathbf{u}_j(c) + \boldsymbol{\xi}_{j,c,\ell},
\qquad
\boldsymbol{\xi}_{j,c,\ell} \sim \mathcal{N}(\mathbf{0},\,0.9^2 \mathbf{I}),
\qquad
\ell = 1,\ldots,n_{jc},
\]
and construct the estimator
\begin{equation*}
\widehat{\mathbf{u}}_j(c) = \frac{1}{n_{jc}} \sum_{\ell=1}^{n_{jc}} \mathbf{X}_{j,c,\ell}.
\end{equation*}
The counts \(n_{jc}\) are drawn from a Dirichlet--multinomial distribution, which induces category imbalance and heterogeneous estimation noise. Collecting the row-wise estimates \(\widehat{\mathbf{u}}_j(c)\) yields the matrix \(\widehat{\mathbf{U}}^{(j)}\).

Given a total embedding budget \(B\), each allocation rule selects dimensions \((d_1,\ldots,d_m)\) satisfying
\begin{equation*}
\sum_{j=1}^m N_j d_j \le B.
\end{equation*}
For each predictor \(j\), the estimated matrix \(\widehat{\mathbf{U}}^{(j)}\) is truncated to rank-\(d_j\) by singular value decomposition, yielding \(\widehat{\mathbf{U}}^{(j)}_{d_j}\).

The response is generated according to the additive model
\begin{equation*}
y = \sum_{j=1}^m \langle \mathbf{u}_j(c_j), \boldsymbol{\beta}_j \rangle + \varepsilon,
\qquad
\varepsilon \sim \mathcal{N}(0,\,0.5^2),
\end{equation*}
where the coefficient vectors \(\boldsymbol{\beta}_j \in \mathbb{R}^{r_j}\) are aligned with the leading singular directions of \(\mathbf{U}^{(j)}\). Performance is evaluated by the out-of-sample MSE on an independent test set. All reported results are averaged over 300 Monte Carlo replications.

We compare four allocation rules, namely the proposed law derived from the optimization problem in Section \ref{sec:method}, equal allocation, proportional-to-cardinality allocation and proportional-to-spectral-mass allocation. Figure~\ref{fig:exp3_performance} reports the resulting test error as a function of the total budget. The theory-based law achieves the lowest, or nearly lowest, test error across the considered budgets, with its advantage being most visible in the low- and intermediate-budget regimes. In these regimes, the budget constraint is active, so suboptimal dimension allocation leads to a substantial loss in predictive accuracy. Equal allocation ignores heterogeneity across predictors and therefore distributes capacity inefficiently when approximation profiles differ. The cardinality-based law assigns excessive capacity to predictors with large \(N_j\), even when those predictors are costly to estimate, and this leads to consistently weaker performance. The spectral-mass rule is more competitive because it captures aggregate signal strength, but it does not explicitly account for the estimation cost induced by predictor cardinality. 

\begin{figure}[t]
\centering
\includegraphics[width=0.85\textwidth]{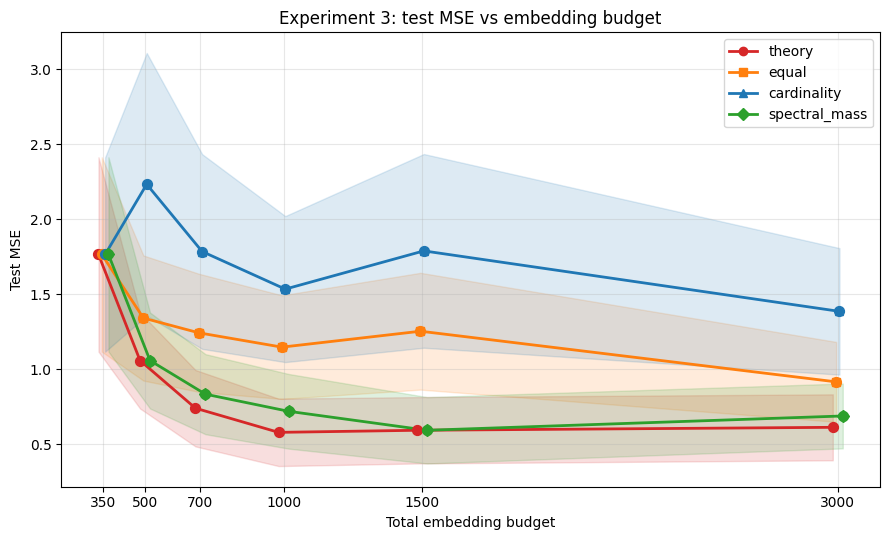}
\caption{\small Out-of-sample MSE as a function of the total embedding budget. Shaded regions indicate variability across replications.}
\label{fig:exp3_performance}
\end{figure}

\begin{figure}[t]
\centering
\includegraphics[width=0.85\textwidth]{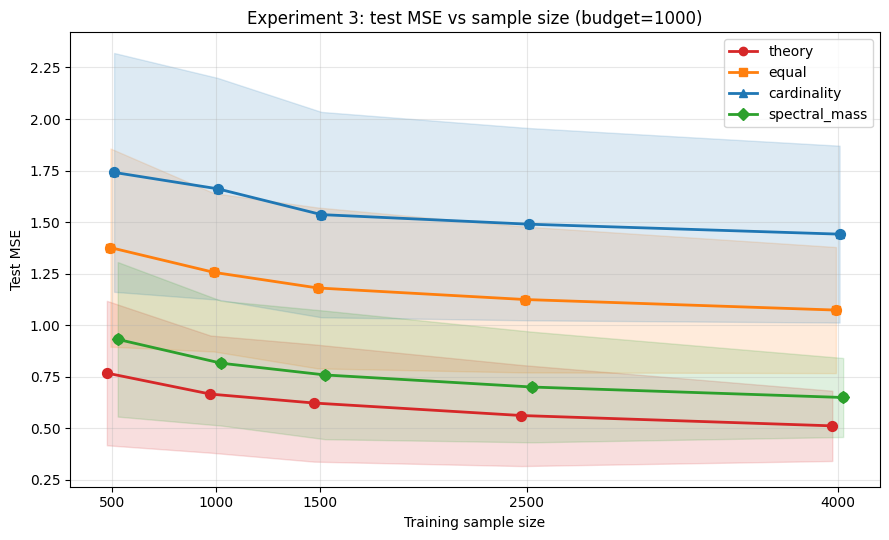}
\caption{\small Out-of-sample MSE as a function of the sample size. Shaded regions indicate variability across replications.}
\label{fig:exp3_performance_samplesize}
\end{figure}

To further assess the finite-sample behavior of the allocation rule, we repeated the multi-predictor allocation experiment over different training sample sizes while keeping the embedding budget fixed at \(B=1000\). Specifically, we considered five training sample sizes and compared the proposed allocation rule with the equal, cardinality-based and spectral mass-based allocation rules under the same data generating mechanism as in the previous experiment. This additional analysis directly evaluates how the performance of the allocation rule changes as the amount of data available for estimating the embeddings increases.

Figure~\ref{fig:exp3_performance_samplesize} shows that the proposed allocation rule achieves the lowest average test MSE across all considered sample sizes. The gain is largest relative to the cardinality-based and equal-allocation benchmarks, which either overemphasize predictor cardinality or ignore predictor heterogeneity altogether. The spectral-mass rule is again competitive, but remains less effective than the proposed rule. As the training sample size increases, all methods improve, but the ranking of the methods remains stable. Finally, Figure~\ref{fig:exp3_allocations} reports the corresponding allocation profiles for different budgets.

\begin{figure}[t]
\centering
\includegraphics[width=1\textwidth]{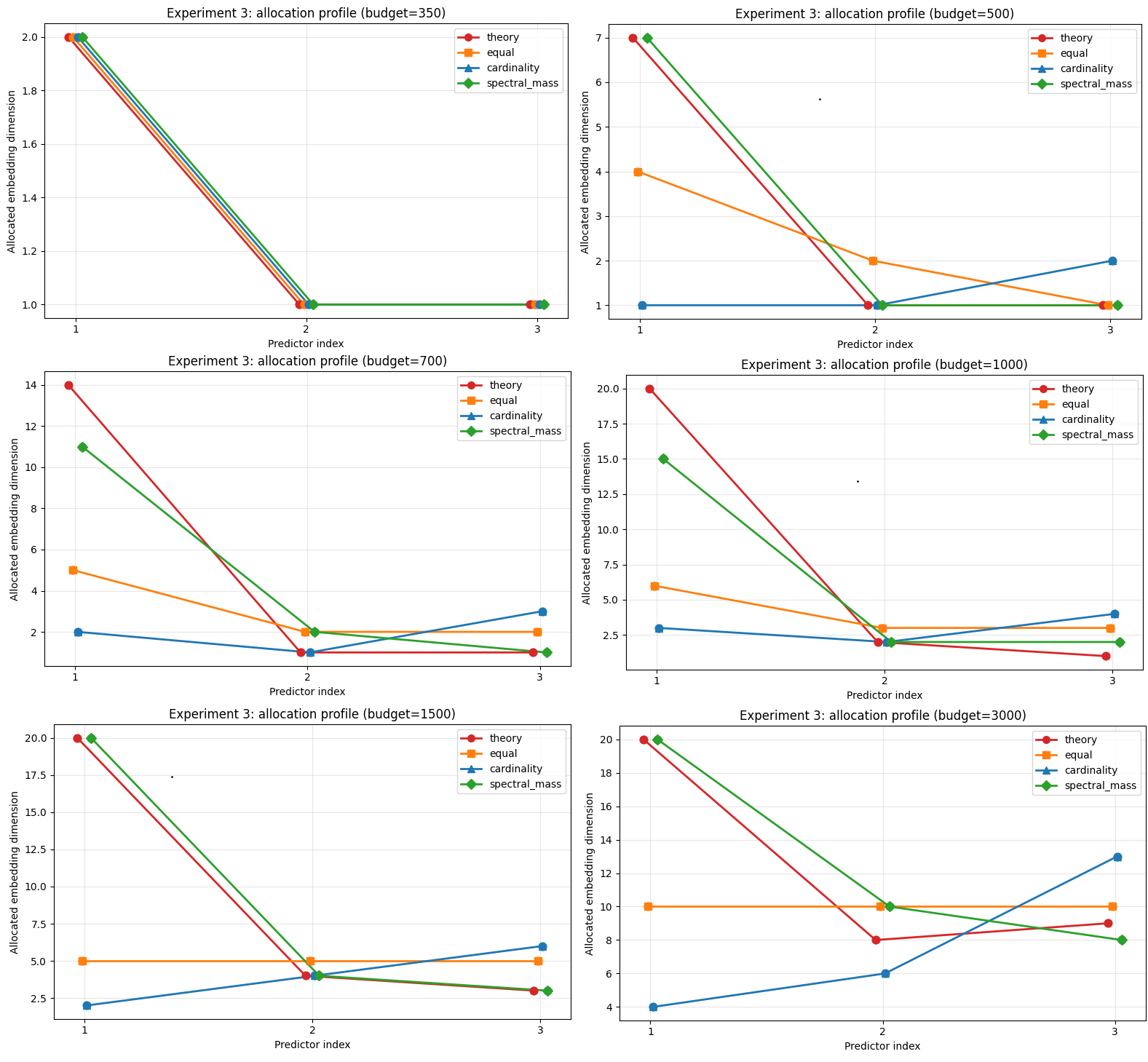}
\caption{\small Allocation of embedding dimensions across predictors for different budgets.}
\label{fig:exp3_allocations}
\end{figure}

\vspace{1em}
\subsection{Application to real-data}

We complement the simulation study with a real-data application on a heart disease prediction task. The purpose of this experiment is not primarily to maximize predictive performance on a specific benchmark. Instead, it is to examine how the proposed embedding allocation rule behaves in a realistic tabular classification setting that combines heterogeneous categorical variables. 

The data used for the purposes of the present analysis are included in the Heart Disease Dataset \citep{Lapp2019}. It contains a binary response indicating the presence or absence of heart disease. The predictors used in our experiment include both categorical and numerical variables. The categorical variables are sex, chest pain type, fasting blood sugar, resting electrocardiographic results, exercise induced angina, the slope of the peak exercise ST segment, the number of major vessels colored by fluoroscopy, thalassemia status, and the data source. The numerical variables are age, resting blood pressure, serum cholesterol, maximum heart rate achieved and ST depression induced by exercise relative to rest. Missing categorical values are imputed with the most frequent category and missing numerical values are imputed with the median. Numerical predictors are standardized. For embedding-based models, categorical variables are ordinal encoded with an additional code reserved for unseen categories at test time. 

We compare five MLPs with different embedding dimension allocation methods. The proposed budget allocation MLP is the primary model and determines feature-specific embedding dimensions under a fixed global parameter budget using the proposed allocation rule. The uniform budget MLP uses the same total budget, but distributes embedding dimensions evenly across categorical predictors. The cardinality heuristic MLP is also constrained by the selected total budget, but allocates additional embedding dimensions using only predictor cardinalities. Starting from the minimum feasible dimension, each additional dimension is assigned to the predictor with the largest cardinality-based priority score $\frac{\sqrt{\max(N_j-1,1)}}{N_j}$,
where \(N_j\) denotes the embedding cardinality of predictor \(j\), including the reserved code for unseen categories. The large unconstrained MLP assigns predictor \(j\) an embedding dimension $d_j = \left\lceil \sqrt{N_j} \right\rceil$,
without enforcing the global budget constraint. This benchmark serves as a practical high-capacity reference.  In this empirical section, \(N_j\) denotes the effective embedding cardinality after ordinal encoding, including the reserved code for unseen categories. Finally, the one-hot MLP replaces learned embeddings with one-hot encoded inputs followed by the same downstream feedforward architecture, serving as a non-embedding baseline.

To isolate the effect of the embedding allocation rule, the compared five MLPs share the same network architecture. In particular, we employ hidden layer widths of $(32, 8)$, ReLU activation functions, and a dropout rate per layer of $0.25$. Data are split into training and test sets using a stratified partition with $70\%$ of the observations for training and $30\%$ for testing \citep{Papageorgiou2025t}. Budget selection is carried out through a holdout validation procedure over the candidate budgets $\mathcal{B} = \{128, 256, 384, 512, 768, 1024, 1536, 2048, 3072, 4096\}$. For each candidate budget, we first estimate feature-specific approximation coefficients from a pilot embedding model. The pilot model uses initial embedding dimensions given by a square root cardinality law. After training, the singular values of each learned embedding matrix are used to construct an empirical coefficient that reflects the extent to which additional embedding dimensions may still be beneficial. These estimated coefficients, together with empirical feature sample sizes, are then fed into the allocation rule to obtain a budget feasible dimension vector. A final model is trained for each candidate budget and the budget is selected using validation logloss with a tolerance law that favors the smallest budget whose validation logloss is within a prescribed neighborhood of the best observed value.
All models are trained with Adam optimizer, learning rate $10^{-3}$, weight decay $10^{-4}$, batch size $32$, and a maximum of $50$ epochs with $10$ epochs as patience \citep{Papageorgiou2022b,Papageorgiou2025s}.

The real-data experiment uses the same embedding-based representation framework and the same allocation principle developed in the theoretical section. The theoretical model provides a structured regularized formulation in which feature-specific embedding matrices, prediction-layer parameters, regularization, and a finite embedding budget are explicitly represented. The empirical MLPs implement this same structure within a practical neural-network pipeline, using dropout, weight decay, and early stopping as regularization mechanisms for finite-sample control. 

Overall, the real-data experiment is designed to answer a focused methodological question. Given a fixed downstream MLP architecture and a realistic mix of categorical and numerical variables, does a principled allocation of embedding dimensions provide a more sensible use of model capacity than uniform allocation, cardinality driven heuristics, unconstrained embedding growth, or the removal of embeddings altogether.

\begin{table}[htbp]\scriptsize
\centering
\caption{Predictive, calibration and computational performance of the compared MLPs on the heart disease prediction task
}
\label{tab:model_comparison_runtime}
\begin{tabular}{lcccccccc}
\hline
Model & Acc & F1 & Prec & MCC & Val Logloss & Brier & ECE & Time/Epoch \\
\hline
\textbf{Proposed MLP}    & \textbf{0.837} & \textbf{0.851} & \textbf{0.860} & \textbf{0.671} & \textbf{0.400} & \textbf{0.121} & \textbf{0.037} & \textbf{0.054} \\
One-Hot MLP      & 0.804 & 0.824 & 0.824 & 0.604 & 0.517 & 0.166 & 0.158 & 0.029 \\
Unconstrained MLP       & 0.815 & 0.833 & 0.836 & 0.620 & 0.472 & 0.148 & 0.125 & 0.059 \\
Uniform Budget MLP     & 0.819 & 0.841 & 0.820 & 0.632 & 0.553 & 0.182 & 0.207 & 0.057 \\
Cardinality Heuristic MLP & 0.757 & 0.756 & 0.852 & 0.534 & 0.526 & 0.173 & 0.105 & 0.072 \\
\hline
\end{tabular}
\end{table}

Table~\ref{tab:model_comparison_runtime} shows that the proposed budget allocation MLP delivers the strongest overall performance among the compared architectures. It achieves the highest accuracy at 0.837, the highest F1 score at 0.851, the highest precision at 0.860, and the highest MCC at 0.671. It also attains the best probabilistic performance, with the lowest validation logloss at 0.4, the lowest Brier score at 0.121 and the lowest ECE at 0.037.

As expected, the one-hot MLP is the fastest model, with a running time per epoch of 0.029. However, its predictive and calibration performance is clearly weaker. The unconstrained and uniform budget MLPs remain competitive, but both are consistently dominated by the proposed model. For example, the unconstrained MLP reaches accuracy 0.815 and MCC 0.620, while the uniform budget MLP reaches accuracy 0.819 and MCC 0.632, both below the corresponding values of the proposed model. The cardinality heuristic MLP performs worst overall. It records the lowest accuracy at 0.757, the lowest F1 score at 0.756, and a substantially lower MCC at 0.534. Consequently, the results indicate that the proposed allocation rule provides the best balance between predictive accuracy, calibration quality and computational cost. Finally, Figure \ref{fig:final_allocation} shows the allocation of embedding dimensions across the $9$ categorical predictors of the dataset.

\begin{figure}[htbp]
    \centering
    \includegraphics[width=0.8\textwidth]{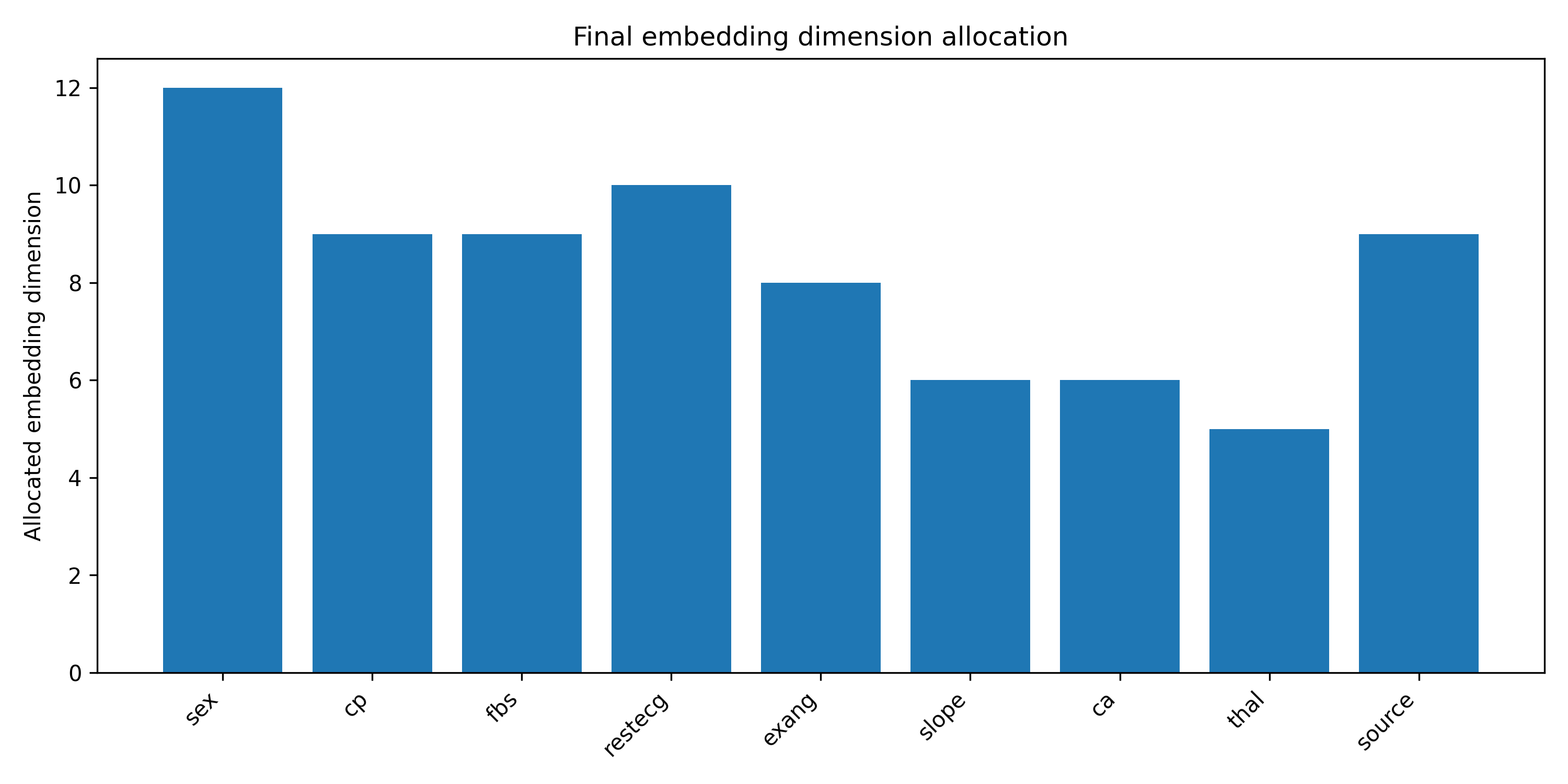}
    \caption{Final embedding dimension allocation across the categorical predictors. The figure reports the selected embedding dimension for each predictor under the proposed allocation rule.}
    \label{fig:final_allocation}
\end{figure}

Overall, the empirical findings are consistent with the theoretical predictions developed in Section 2. The proposed allocation rule yields systematic improvements in both predictive accuracy and probabilistic calibration while maintaining comparable computational cost relative to the competing predictive models. These results suggest that embedding dimension assignment, when treated as a constrained allocation problem, leads to an efficient use of representational capacity in finite-sample settings. 

\section{Discussion and Conclusions}\label{sec:concl}

In this paper, we address a design question that is common in representation learning practice but rarely formulated as an explicit methodological problem. When embeddings are introduced under finite data, the central issue is not simply whether low-dimensional representations should be used, but how limited representational capacity should be distributed across heterogeneous components. In many applications, the total representational budget is constrained by sample size, statistical stability, computational cost and implementation requirements. Under such conditions, embedding dimension assignment is not a purely expressive modeling choice, but a constrained design decision that must balance flexibility against finite-sample estimation burden.

The main contribution of the paper is to formulate embedding dimension assignment as a constrained allocation problem governed by an explicit approximation-estimation tradeoff. The analysis shows that dimensions should not be chosen uniformly or solely on the basis of simple measures, such as feature cardinality. Instead, they should be allocated according to the marginal value of additional representational flexibility under finite-sample constraints. In this way, the paper replaces a heuristic modeling choice with a mathematically explicit allocation principle, so the contribution is structural rather than merely algorithmic.

A second contribution is conceptual. In much of the practical literature on ML and DL, embedding dimensions are chosen primarily for approximation quality, while estimation effects are handled only indirectly through validation, regularization, or downstream tuning. By incorporating estimation cost directly into the design objective, the present framework recasts embedding dimension choice as a finite-sample resource allocation problem. This perspective places representation design within the broader class of constrained allocation problems and makes the problem more interpretable and analytically tractable.

The theoretical results are complemented by computational evidence that clarifies when the proposed methodology matters most. The simulation study shows that informed allocation yields the greatest benefit precisely in regimes where finite-sample effects are strongest and component heterogeneity is most pronounced, which are also the regimes in which naive laws are least reliable. The real-data experiment emphasizes that non-uniform embedding design can improve predictive performance in a practically relevant tabular setting under limited representational capacity. Together, these results tie the theory to the operating region where allocation choices are consequential rather than to high-capacity settings where they matter less, and they provide empirical support for the paper’s interpretation of embedding design as a finite-sample resource allocation problem.

The empirical conclusions should be interpreted with appropriate scope. The healthcare application provides evidence from one realistic tabular classification setting, but it does not imply that the same performance gains will necessarily occur across all domains or against all non-MLP baselines. The proposed framework is expected to be most useful when categorical predictors are heterogeneous and the sample size is limited. In such cases, different predictors may have different approximation value and different parameter cost, so uniform or cardinality-based choices allocate capacity inefficiently. When the available budget is large, or when strong non-embedding tabular methods already capture the relevant categorical structure, the practical gains may be smaller.

Moreover, the proposed rule is most appropriate when the empirical spectra are stable, categories are sufficiently observed, and approximation differences dominate estimation noise. Conversely, the empirical singular-value tail may be less reliable as a proxy for approximation value when the learned embeddings are noisy (e.g. in small-sample settings), when some categories are sparsely observed, or when the empirical singular-value tail does not display a stable monotone decay pattern. In such cases, the estimated coefficient should be interpreted cautiously as a data-driven allocation score rather than as an exact population quantity.

The same principle may also be relevant in higher capacity deep learning (DL) models, such as convolutional neural network (CNN) pipelines, where high-dimensional features derived from images are combined with categorical variables. A characteristic example would be a unified DL model designed to detect tumor presence from images of different organs, where the organ type is included as a categorical variable. In such settings, even for low-complexity CNNs \citep{Papageorgiou2025n}, standard one-hot encodings are expected to contribute weakly relative to the high-dimensional image feature vector. By contrast, learned categorical embeddings with principled dimension allocation can provide a more compact and adaptive representation.

A few points nevertheless delimit the scope of the present analysis and indicate natural directions for future work. The framework is built on a decomposition of risk into approximation and estimation terms so that the allocation problem can be studied in a transparent and tractable form. This abstraction makes it possible to derive explicit structural results and to isolate the core finite-sample allocation mechanism in a transparent way. However, several practically relevant features remain unexplored, including optimization effects, richer dependence across feature groups and shared latent structure that may induce additional coupling across components \citep{Huang2021,Salter-Townshend2017}. Addressing these elements will require models that jointly capture statistical error, optimization dynamics and cross-component dependence, together with empirical procedures for estimating these quantities and validating adaptive allocation rules in realistic data settings.

\section*{Data Availability}
The Heart Disease Dataset used in the present analysis is publicly available and may be obtained from \url{https://www.kaggle.com/datasets/johnsmith88/heart-disease-dataset}.

\section*{Declaration of competing interests}
The author declares that he has no known competing financial interests or personal relationships that could have appeared to influence the work reported in this paper.

\bibliographystyle{apalike}
\bibliography{export.bib}

\end{document}